\documentclass{amsart}

\usepackage{amsmath}
\usepackage{amsfonts}
\usepackage{amssymb,enumerate}
\usepackage{amsthm}
\usepackage[all]{xy}
\usepackage{hyperref}

\newtheorem{lem}{Lemma}[section]
\newtheorem{cor}[lem]{Corollary}
\newtheorem{prop}[lem]{Proposition}
\newtheorem{thm}[lem]{Theorem}

\newtheorem{Defn}[lem]{Definition}
\newtheorem{Ex}[lem]{Example}
\newtheorem{Question}[lem]{Question}
\newtheorem{Property}[lem]{Property}
\newtheorem{Properties}[lem]{Properties}
\newtheorem{Discussion}[lem]{Remark}
\newtheorem{Construction}[lem]{Construction}
\newtheorem{Notation}[lem]{Notation}
\newtheorem{Fact}[lem]{Fact}
\newtheorem{Notationdefinition}[lem]{Definition/Notation}
\newtheorem{Remarkdefinition}[lem]{Remark/Definition}
\newtheorem{Subprops}{}[lem]
\newtheorem{Para}[lem]{}

\newenvironment{defn}{\begin{Defn}\rm}{\end{Defn}}
\newenvironment{ex}{\begin{Ex}\rm}{\end{Ex}}
\newenvironment{question}{\begin{Question}\rm}{\end{Question}}

\newenvironment{notation}{\begin{Notation}\rm}{\end{Notation}}
\newenvironment{fact}{\begin{Fact}\rm}{\end{Fact}}

\newenvironment{disc}{\begin{Discussion}\rm}{\end{Discussion}}

\newtheorem{intthm}{Theorem}

\newtheorem{intQ}[intthm]{Question}
\newenvironment{intq}{\begin{intQ}\rm}{\end{intQ}}
\newtheorem{intEx}[intthm]{Example}
\newenvironment{intex}{\begin{intEx}\rm}{\end{intEx}}
\newcommand{\cat}[1]{\mathsf{#1}}
\newcommand{\catd}{\mathsf{D}}

\newcommand{\catm}{\cat{M}}

\newcommand{\catc}{\cat{C}}

\newcommand{\pd}{\operatorname{pd}}	
	
\newcommand{\gkdim}[1]{\mathrm{G}_{#1}\text{-}\!\dim}	
\newcommand{\gcdim}{\gkdim{C}}	
	
\newcommand{\id}{\operatorname{id}}

\newcommand{\depth}{\operatorname{depth}}	
\newcommand{\rank}{\operatorname{rank}}

\newcommand{\ann}{\operatorname{Ann}}
\newcommand{\ext}{\operatorname{Ext}}	
\newcommand{\rhom}{\mathbf{R}\!\operatorname{Hom}}	
\newcommand{\lotimes}{\otimes^{\mathbf{L}}}
\newcommand{\HH}{\operatorname{H}}
\newcommand{\Hom}{\operatorname{Hom}}	

\newcommand{\spec}{\operatorname{Spec}}

\newcommand{\s}{\mathfrak{S}}
\newcommand{\tor}{\operatorname{Tor}}
\newcommand{\im}{\operatorname{Im}}
\newcommand{\shift}{\mathsf{\Sigma}}

\newcommand{\Pic}{\operatorname{Pic}}

\newcommand{\Ker}{\operatorname{Ker}}

\newcommand{\ideal}[1]{\mathfrak{#1}}
\newcommand{\m}{\ideal{m}}
\newcommand{\n}{\ideal{n}}

\newcommand{\comp}[1]{\widehat{#1}}
\newcommand{\ol}{\overline}

\newcommand{\ass}{\operatorname{Ass}}
\newcommand{\supp}{\operatorname{Supp}}

\newcommand{\bbz}{\mathbb{Z}}

\newcommand{\xra}{\xrightarrow}

\newcommand{\res}{\xra{\simeq}}

\newcommand{\cl}[1]{[#1]}

\newcommand{\tri}{\trianglelefteq}
\newcommand{\trin}{\triangleleft}

\newcommand{\rhat}{\widehat R}
\newcommand{\bhat}{\widehat B}
\newcommand{\chat}{\widehat C}
\newcommand{\ghat}{\widehat G}

\renewcommand{\geq}{\geqslant}
\renewcommand{\leq}{\leqslant}
\renewcommand{\ker}{\Ker}
\renewcommand{\hom}{\Hom}

\numberwithin{equation}{lem}

\begin{document}

\bibliographystyle{amsplain}

\author{Sean Sather-Wagstaff}

\address{Department of Mathematics,
NDSU Dept \# 2750,
PO Box 6050,
Fargo, ND 58108-6050
USA}

\email{Sean.Sather-Wagstaff@ndsu.edu}

\urladdr{http://math.ndsu.nodak.edu/faculty/ssatherw/}

\title{Bass Numbers and Semidualizing Complexes}


\begin{abstract}
Let $R$ be a commutative local noetherian ring.
We prove that the existence of a  chain  of semidualizing $R$-complexes 
of length $(d+1)$ yields a degree-$d$ polynomial lower bound for the
Bass numbers of $R$. We also show how information about certain
Bass numbers of $R$ provide restrictions on the lengths of chains of semidualizing 
$R$-complexes. To make this article somewhat self-contained, we also include a
survey of some of the basic properties of semidualizing modules, semidualizing complexes
and derived categories.
\end{abstract}

\maketitle

\section*{Introduction}

Throughout this paper
$(R,\m,k)$ is a commutative local noetherian ring.

A classical maxim from module theory states that
the existence of certain types of $R$-modules forces 
ring-theoretic conditions on $R$. For instance, if
$R$ has a dualizing module, then $R$ is Cohen-Macaulay 
and a homomorphic image of a Gorenstein ring. 

This paper is concerned with the consequences of the existence of 
nontrivial \emph{semi}-dualizing $R$-modules and, more generally,
semidualizing $R$-complexes. 
In this introduction, we restrict our attention to the modules.
Essentially, a semidualizing module
differs from a dualizing module in that the semidualizing module is
not required to have finite injective dimension. (See Section~\ref{sec12}
for definitions and background information.)
The set of isomorphism classes of semidualizing $R$-modules 
has a  rich structure. For instance, it
comes equipped with an ordering based on the notion of 
total reflexivity. 

It is not clear that the existence of nontrivial
semidualizing $R$-complexes should have any deep implications for  $R$.
For instance, every ring has at least one semidualizing $R$-module,
namely, the free $R$-module of rank 1.
However, Gerko~\cite{gerko:sdc} has shown that, when $R$ is artinian, the 
existence of certain collections of semidualizing 
$R$-modules implies the existence of a lower bound
for the Loewy length of $R$; moreover, if this lower bound is achieved,
then the Poincar\'e series of $k$ has a very specific form.

The first point of this paper is to show how the existence of nontrivial
semidualizing modules gives some insight into the following questions of Huneke
about the \emph{Bass numbers} $\mu^i_R(R)=\rank_k(\ext^i_R(k,R))$.

\begin{intq} \label{q0001}
Let $R$ be a local Cohen-Macaulay ring.
\begin{enumerate}[\quad(a)]
\item \label{q0001a}
If the sequence $\{\mu^i_R(R)\}$ is bounded, must it be eventually 0? that is, must $R$ be Gorenstein?
\item \label{q0001b}
If the sequence $\{\mu^i_R(R)\}$ is bounded above by a polynomial in $i$, must $R$ be Gorenstein?
\item \label{q0001c}
If $R$ is not Gorenstein, must the sequence $\{\mu^i_R(R)\}$ grow exponentially?
\end{enumerate}
\end{intq}

Some progress on these questions has been made
by 
Borna Lorestani, Sather-Wagstaff and Yassemi~\cite{yassemi:vhgbscm},
Christensen, Striuli and Veliche~\cite{christensen:iabbn}, and
Jorgensen and Leuschke~\cite{jorgensen:gbscm}.
However, each of these questions is still open in general.
The following result gives the connection with semidualizing modules. 
It is contained in Theorem~\ref{thm0101} and Corollary~\ref{cor0101}.
Note that this result does not assume that $R$ is Cohen-Macaulay.

\begin{intthm} \label{thm0001}
Let $R$ be a local ring.
If $R$ has a semidualizing module that is neither dualizing nor free,
then the sequence  of Bass numbers $\{\mu^i_R(R)\}$ is bounded below by a linear polynomial in $i$
and hence
is not eventually constant.
Moreover, if $R$ has a chain of semidualizing modules of length $d+1$,
then the sequence  of Bass numbers $\{\mu^i_R(R)\}$ is bounded below by a polynomial in $i$
of degree $d$.
\end{intthm}

\noindent
For readers who are familiar with semidualizing modules, the proof of this result
is relatively straightforward when $R$ is Cohen-Macaulay. 
We outline the proof here. 
Pass to the completion of $R$ in order to assume that $R$ is complete,
and hence has a dualizing module $D$.
The Bass series $I_R^R(t)$ of $R$  then agrees with  the Poincar\'{e} series
$P^R_D(t)$ of $D$, up  to a shift. Because of a result of Gerko~\cite[(3.3)]{gerko:sdc}
the given chain of semidualizing modules yields a factorization 
$P^R_D(t)=P_1(t)\cdots P_{d+1}(t)$ where each $P_i(t)$ is a power
series with positive integer coefficients. 
The result now follows from
straightforward numerics. 
The proof in the general case is essentially the
same: after passing to the completion, use semidualizing \emph{complexes} and
the Poincar\'{e} series of a dualizing complex for $R$.

The second point of this paper is to show how information about 
certain Bass numbers of $R$ force restrictions on the set of isomorphism
classes 
of semidualizing $R$-modules. 
By way of motivation, we recall one of the main open questions about this set:
must it be finite? Christensen and Sather-Wagstaff~\cite{christensen:cmafsdm}
have made some progress on this question,
but the general question is still open. While the current paper
does not address this question directly, we do show that
this set cannot contain chains of arbitrary length under the reflexivity ordering.
This is contained in 
the next result which summarizes Theorems~\ref{prop0101} and~\ref{prop0103}.
Note that the integer $\mu^g_R(R)$ in part~\eqref{thm0002b} is the Cohen-Macaulay type
of $R$.

\begin{intthm} \label{thm0002}
Let $R$ be a local ring of depth $g$.
\begin{enumerate}[\quad\rm(a)]
\item \label{thm0002a}
If $R$ has a chain of semidualizing modules of length $d$,
then $d\leq\mu^{g+1}_R(R)$. So, the ring $R$ does not have
arbitrarily long chains of semidualizing modules.
\item \label{thm0002b}
Assume that $R$ is Cohen-Macaulay. 
Let $h$ denote the number of prime factors of the integer
$\mu^g_R(R)$, counted with multiplicity.
If $R$ has a chain of semidualizing modules of length $d$,
then 
$d\leq h\leq\mu^g_R(R)$.
In particular, if $\mu^g_R(R)$
is prime, then every semidualizing $R$-module
is either free or dualizing for $R$.
\end{enumerate}
\end{intthm}

\noindent
As an introductory  application of these ideas, we have the following:

\begin{intex} \label{ex0001}
Let $k$ be a field and set $R=k[\![X,Y]\!]/(X^2,XY)$. 
For each semidualizing $R$-module $C$,
one has $C\cong R$.
Indeed, 
the semidualizing property implies that
$\beta^R_0(C)\mu^0_R(C)=\mu^0_R(R)=1$
where $\beta^R_0(C)$ is the minimal number of generators of $C$.
It follows that $C$ is cyclic, so $C\cong R/\ann_R(C)\cong R$.
See Facts~\ref{f0201'} and~\ref{f0220'}.
\end{intex}

We prove more facts about the semidualizing objects for this ring in Example~\ref{ex0601}.

We now summarize the contents of and philosophy behind this paper.
Section~\ref{sec12} contains the basic properties of semidualizing modules
needed for the proofs of Theorems~\ref{thm0001} and~\ref{thm0002}.
Section~\ref{sec02} outlines the necessary background 
on semidualizing complexes needed for the more general versions of
Theorems~\ref{thm0001} and~\ref{thm0002}, which are the subjects of
Sections~\ref{sec05} and~\ref{sec06}.
Because the natural habitat for semidualizing complexes is the derived category $\catd(R)$,
we include a brief introduction to this category in Appendix~\ref{sec01} for
readers who desire some background.

Sections~\ref{sec12} and~\ref{sec02} are arguably longer than necessary
for the proofs of the results of Sections~\ref{sec05} and~\ref{sec06}.
Moreover, Section~\ref{sec12} is essentially a special case of
Section~\ref{sec02}. This is justified by the third point of this paper:
We hope that, after seeing our applications
to Question~\ref{q0001}, some readers will be motivated
to learn more about semidualizing objects. To further encourage this, 
Section~\ref{sec12} is a brief survey of the theory for modules.
We hope this will be helpful for readers who are familiar with 
dualizing modules, but possibly not familiar with dualizing complexes.

Section~\ref{sec02} is a parallel survey 
of the more general 
semidualizing complexes. It is written for readers who are familiar
with dualizing complexes and
the category of chain complexes and who have at least some knowledge
about the derived category. 

For readers who find their background on
the derived category lacking, Appendix~\ref{sec01} contains
background material on this subject. Our hope is
to impart enough information about this category so that 
most readers get a feeling for the ideas behind our proofs.
As such, we stress the connections between this category and the category
of $R$-modules.

\section{Semidualizing Modules} \label{sec12}

This section contains an introduction to our main players
when they are modules.
These are the semidualizing modules, which were introduced independently
(with different terminology) by
Foxby~\cite{foxby:gmarm},
Golod~\cite{golod:gdagpi},
Vasconcelos~\cite{vasconcelos:dtmc} and
Wakamatsu~\cite{wakamatsu:mtse}.
They generalize Grothendieck's notion of a dualizing module~\cite{hartshorne:lc}
and encompasses duality theories with respect to dualizing modules and
with respect to the ring $R$.

\begin{defn} \label{d0201'}
Let $C$ be an $R$-module. 
The \emph{homothety homomorphism} associated to $C$
is the $R$-module homomorphism
$\smash{\ol\chi}^R_C\colon R\to\hom_R(C,C)$  given by
$\smash{\ol\chi}^R_C(r)(c)=rc$. 
The $R$-module $C$ is \emph{semidualizing} if
it satisfies the following conditions:
\begin{enumerate}[\quad(1)]
\item  \label{d0201a'}
The $R$-module $C$ is finitely generated;
\item  \label{d0201b'}
The homothety map $\smash{\ol\chi}^R_C\colon R\to\hom_R(C,C)$
is an isomorphism; and
\item  \label{d0201c'}
For all $i\geq 1$, we have $\ext^i_R(C,C)=0$. 
\end{enumerate}
An $R$-module $D$ is \emph{dualizing} if  it is semidualizing and
has finite injective dimension.

The set of isomorphism classes of semidualizing
$R$-modules is denoted $\s_0(R)$, and the isomorphism
 class of a given semidualizing $R$-module $C$ is
denoted $\cl C$. 
\end{defn}

\begin{ex} \label{ex0201'}
The $R$-module $R$ is semidualizing, so $R$  has a semidualizing module.
\end{ex}

\begin{disc} \label{d9999}
For this article, we have assumed that the ring $R$ is local.
While this assumption is not necessary for the definitions and basic properties of semidualizing
modules, it does make the theory somewhat simpler. 

Specifically, let $S$ be a commutative noetherian ring, 
not necessarily local, and let $C$ be an $S$-module.
Define the homothety homomorphism $\smash{\ol\chi}^S_C\colon S\to\hom_S(C,C)$,
the semidualizing property, and the set $\s_0(S)$ 
as in~\ref{d0201'}. 
It is straightforward to show that the semidualizing property is local, that is, that $C$ is a semidualizing
$S$-module if and only if $C_{\n}$ is a semidualizing $S_{\n}$-module for each maximal ideal
$\n\subset S$. For instance, every finitely generated projective $S$-module of rank 1 is
semidualizing. In other words, the Picard group $\Pic(S)$ is a subset of  $\s_0(S)$.
Also, the group $\Pic(S)$ acts on
$\s_0(S)$ in a natural way: for each semidualizing $S$-module $C$ and each 
finitely generated projective $S$-module $L$ of rank 1, the $S$-module $L\otimes_SC$
is semidualizing. 
This action is trivial when $S$ is local as the Picard group of a local ring contains 
only the free module of rank 1.

While this gives the nonlocal theory more structure to investigate,
one can view this additional structure as problematic, for the following reason.
Fix a
semidualizing $S$-module $C$ and a 
finitely generated projective $S$-module $L$ of rank 1.
Define the terms ``totally $C$-reflexive'' and ``totally $L\otimes_SC$-reflexive''
as in~\ref{d0202'}. It is straightforward to show that an $S$-module
$G$ is totally $C$-reflexive if and only if it is totally $L\otimes_SC$-reflexive.
In particular, when $\Pic(S)$ is nontrivial, the reflexivity ordering on $\s_0(S)$,
defined as in~\ref{d0204'}, is not antisymmetric. Indeed,  one has
$\cl{C}\tri\cl{L\otimes_SC}\tri\cl{C}$, even though $\cl{C}=\cl{L\otimes_SC}$ if and only if
$L\cong S$. 

One can overcome the lack of antisymmetry by considering the set $\ol{\s_0}(S)$ of
orbits in $\s_0(S)$ under the Picard group
action. (Indeed, investigations of $\ol{\s_0}(S)$ can be found in the work of
Avramov, Iyengar, and Lipman~\cite{avramov:rrc1} and
Frankild, Sather-Wagstaff and Taylor~\cite{frankild:rbsc}.)
However, we choose to avoid this level of generality in the current paper,
not only for the sake of simplicity, but also 
because our applications in Section~\ref{sec05} and~\ref{sec06}
are explicitly for local rings.

For the record, we note that
another level of complexity arises when the ring $S$ is of the form $S_1\times S_2$
where $S_1$ and $S_2$ are (nonzero) commutative noetherian rings.
In this setting, the semidualizing $S$-modules are all of the form
$C_1\oplus C_2$ where each $C_i$ is a semidualizing $S_i$-module.
In other words, each connected component of $\spec(S)$ contributes
a degree of freedom to the elements of $\s_0(S)$, and to $\ol{\s_0}(S)$.
For further discussion, see~\cite{frankild:rrhffd,frankild:rbsc}.
\end{disc}

The next three facts contain fundamental properties of semidualizing modules.

\begin{fact} \label{f0201'}
Let $C$ be a semidualizing $R$-module.
The isomorphism $R\cong\hom_R(C,C)$ implies that
$\ann_R(C)=0$. 
It follows that 
$\supp_R(C)=\spec(R)$ and so $\dim(C)=\dim(R)$.
Furthermore $C$ is  cyclic  if and only if $C\cong R$:
for the nontrivial implication, if $C$ is  cyclic,
then $C\cong R/\ann_R(C)\cong R$. 
Thus, if $C\not\cong R$, then $\beta^R_0(C)\geq 2$.
Here $\beta^R_0(C)$ is the 0th Betti number of $C$, i.e.,
the minimal number of generators of $C$.

Furthermore, the isomorphism 
$R\cong\hom_R(C,C)$ also implies that $\ass_R(C)=\ass(R)$.
It follows that an element $x\in \m$ is $C$-regular if and only if it is $R$-regular.
When $x$ is $R$-regular, one can show that the $R/xR$-module $C/xC$
is semidualizing; see~\cite[(4.5)]{frankild:rrhffd}. 
Hence, by induction, we have $\depth_R(C)=\depth(R)$.
\end{fact}

\begin{fact} \label{f0231'}
If $R$ is Gorenstein, then every semidualizing $R$-module is isomorphic 
to $R$; see~\cite[(8.6)]{christensen:scatac} or Theorem~\ref{prop0101}.
(Note that the assumption that $R$ is local is crucial here
because of Remark~\ref{d9999}.)
The converse of this statement holds when
$R$ has a dualizing module by~\cite[(8.6)]{christensen:scatac}; 
the converse can fail when $R$ does not have a dualizing
module by \cite[(5.5)]{christensen:dvke}. Compare this with Fact~\ref{f0201}.
\end{fact}

\begin{fact} \label{f0215'}
A result of 
Foxby~\cite[(4.1)]{foxby:gmarm},
Reiten~\cite[(3)]{reiten:ctsgm}
and
Sharp~\cite[(3.1)]{sharp:gmccmlr} 
says that $R$ has a dualizing module if and only if $R$ is
Cohen-Macaulay and a homomorphic image of a Gorenstein ring.
Hence, if $R$ is complete and Cohen-Macaulay, then Cohen's structure theorem
implies that $R$ has a dualizing module.
Compare this with Fact~\ref{f0215}.
\end{fact}

We next give the first link between semidualizing modules
and Bass numbers.

\begin{fact} \label{f0206'}
Assume that $R$ is Cohen-Macaulay of depth $g$.
If $R$ has a dualizing module $D$, then 
for each $i\geq 0$ we have
$\mu^{i+g}_R(R)=\beta_{i}^R(D)$.
Moreover, if  $D'$ is a dualizing module for $\rhat$,
then 
for each integer $i\geq 0$ we have
$\mu^{i+g}_R(R)=\mu^{i+g}_{\comp R}(\comp R)=\beta_{i}^{\rhat}(D')$;
see e.g.\ \cite[(1.5.3),(2.6)]{avramov:rhafgd} and~\cite[(V.3.4)]{hartshorne:rad}.
Compare this with Fact~\ref{f0206}.
\end{fact}

Here is one of the main open questions in this subject.
An affirmative answer for the case
when $R$ is Cohen-Macaulay and equicharacteristic
is given in~\cite[(1)]{christensen:cmafsdm}.
Note that it is crucial that $R$ be local; see Remark~\ref{d9999}.
Also note that, while Theorem~\ref{prop0103} shows that chains in
$\s_0(R)$ cannot have arbitrarily large length, the methods of this paper
do not answer this question.

\begin{question} \label{q0201'}
Is the set $\s_0(R)$  finite?
\end{question}

The next fact documents some fundamental properties.

\begin{fact} \label{f0204'}
When $C$ is a finitely generated $R$-module, it is semidualizing for $R$
if and only if the completion $\widehat C$ is semidualizing for $\rhat$.
See~\cite[(5.6)]{christensen:scatac}.
The essential point of the proof is that there are isomorphisms
$$\ext^i_{\rhat}(\chat,\chat)\cong\rhat\otimes_R\ext^i_R(C,C).$$
(The analogous result holds for the dualizing property by, e.g., \cite[(3.3.14)]{bruns:cmr}.)
Thus, the assignment $C\mapsto\widehat C$ induces a well-defined
function $\s_0(R)\hookrightarrow\s_0(\rhat)$;
this function is injective since, for finitely generated $R$-modules
$B$ and $C$, we have $B\cong C$ if and only if $\bhat\cong\chat$.
From~\cite[(5.5)]{christensen:dvke} we know that this map can fail to be surjective. 
Compare this with Fact~\ref{f0204}.
\end{fact}

Next we summarize the aspects of duality with respect to semidualizing modules
that are relevant for our results.

\begin{defn} \label{d0202'}
Let $C$  and $G$ be $R$-modules. 
The  \emph{biduality homomorphism} associated to $C$ and $G$ is the map
$\smash{\ol\delta}^C_G\colon G\to\hom_R(\hom_R(G,C),C)$
given by $\smash{\ol\delta}^C_G(x)(\phi)=\phi(x)$.

Assume that $C$ is a semidualizing $R$-module.
The $R$-module $G$ is \emph{totally $C$-reflexive}
when
it satisfies the following conditions:
\begin{enumerate}[\quad(1)]
\item  \label{f0208a'}
The $R$-module $G$ is finitely generated;
\item  \label{f0208b'}
The biduality map
$\smash{\ol\delta}^C_G\colon G\to\hom_R(\hom_R(G,C),C)$
is an isomorphism; and
\item  \label{f0208c'}
For all $i\geq 1$, we have $\ext^i_R(G,C)=0=\ext^i_R(\hom_R(G,C),C)$. 
\end{enumerate}
\end{defn}

\begin{fact} \label{f0208'}
Let $C$ be a semidualizing $R$-module.
It is straightforward to show that every finitely generated free $R$-module  is totally $C$-reflexive.
The essential point of the proof is that there are isomorphisms
\begin{gather*}
\ext^i_R(R^n,C)
\cong\begin{cases} 0 &\text{if $i\neq 0$} \\ C^n & \text{if $i=0$} \end{cases} \\
\ext^i_R(\hom_R(R^n,C),C)
\cong\ext^i_R(C^n,C)\cong\ext^i_R(C,C)^n
\cong\begin{cases} 0 &\text{if $i\neq 0$} \\ R^n & \text{if $i=0$} \end{cases}
\end{gather*}
It follows that every finitely generated $R$-module $M$ has a resolution
by totally $C$-reflexive $R$-modules
$\cdots\to G_1\to G_0\to M\to 0$.
It is similarly straightforward to show that $C$ is totally $C$-reflexive
because
\begin{gather*}
\ext^i_R(C,C)
\cong\begin{cases} 0 &\text{if $i\neq 0$} \\ R & \text{if $i=0$} \end{cases} \\
\ext^i_R(\hom_R(C,C),C)
\cong\ext^i_R(R,C)
\cong\begin{cases} 0 &\text{if $i\neq 0$} \\ C & \text{if $i=0$.} \end{cases}
\end{gather*}
Compare this with Facts~\ref{f0208} and~\ref{f0299}.
\end{fact}

The next definition was introduced by Golod~\cite{golod:gdagpi}.

\begin{defn} \label{d0212'}
Let $C$ be a semidualizing $R$-module, 
and let $M$ be a  finitely generated $R$-module. 
If $M$ has a bounded resolution by totally
$C$-reflexive $R$-modules, then it has 
\emph{finite $\text{G}_C$-dimension}
and its
\emph{$\text{G}_C$-dimension},
denoted $\gcdim_R(M)$ is the length
of the shortest such resolution.
\end{defn}

The next fact contains the
ever-useful ``AB-formula'' for $\text{G}_C$-dimension and 
is followed by some of its consequences.

\begin{fact} \label{f0299'}
Let  $C$ be a semidualizing $R$-module.
If $B$ is an $R$-module of finite $\text{G}_C$-dimension,
then $\gcdim_R(B)=\depth(R)-\depth_R(B)$;
see~\cite[(3.14)]{christensen:scatac} or~\cite{golod:gdagpi}.
\label{f0217'}
When $B$ is  semidualizing, 
Facts~\ref{f0201'} and~\ref{f0299'} combine to show that
$B$ has finite $\text{G}_C$-dimension if and only if $B$ is totally $C$-reflexive.
\end{fact}

\begin{fact} \label{f0205'}
Let $C$ be a semidualizing $R$-module. If $\pd_R(C)<\infty$, then $C\cong R$
Indeed, using Fact~\ref{f0201'}, the Auslander-Buchsbaum formula
shows that $C$ must be free, and the isomorphism $\Hom_R(C,C)\cong R$
implies that $C$ is free of rank 1.
(Note that this depends on the assumption that $R$ is local;
see Remark~\ref{d9999}.)
It follows that, if $C$ is a non-free semidualizing $R$-module,
then the Betti number $\beta_i^R(C)$ is positive for each integer $i\geq 0$.
Compare this with Fact~\ref{f0205} and Lemma~\ref{l0201}.
Questions about the Betti numbers of semidualizing modules
akin to those in Question~\ref{q0001} are contained in~\ref{q}.
\end{fact}

The next facts contain some  fundamental properties of this notion of reflexivity.

\begin{fact} \label{f0210'}
Let $C$ be a semidualizing $R$-module.
A finitely generated $R$-module $G$ is totally $C$-reflexive
if and only if the completion $\widehat G$ is totally $\widehat C$-reflexive.
The essential point of the proof is that there are isomorphisms
\begin{align*}
\ext^i_{\rhat}(\ghat,\chat)
&\cong\rhat\otimes_R\ext^i_R(G,C)\\
\ext^i_{\rhat}(\hom_{\rhat}(\ghat,\chat),\chat)
&\cong\ext^i_{\rhat}(\rhat\otimes_R\hom_R(G,C),\rhat\otimes_RC)\\
&\cong\rhat\otimes_R\ext^i_R(\hom_R(G,C),C).
\end{align*}
Furthermore, a finitely generated $R$-module $M$ 
has finite $\text{G}_C$-dimension if and only if
$\widehat M$ has finite $\text{G}_{\comp C}$-dimension.
See~\cite[(5.10)]{christensen:scatac} or~\cite{golod:gdagpi}.
Compare this with Fact~\ref{f0210}.
\end{fact}

\begin{fact} \label{f0211'}
Let $C$ be a semidualizing $R$-module. If $M$ is a
finitely generated $R$-module of finite projective dimension, then $M$ 
has finite $\text{G}_C$-dimension by Fact~\ref{f0208'}.

Let $D$ be a dualizing $R$-module.
If $M$ is a 
maximal Cohen-Macaulay $R$-module, then $M$ is totally $D$-reflexive
by~\cite[(3.3.10)]{bruns:cmr}. The converse holds 
because of the AB-formula~\ref{f0299'}.
It follows that every finitely generated $R$-module $N$ has finite $\text{G}_D$-dimension,
as the fact that $R$ is Cohen-Macaulay 
(c.f.\ Fact~\ref{f0215'})
implies that some syzygy of $N$ is 
maximal Cohen-Macaulay.
Compare this with Fact~\ref{f0299}.
\end{fact}

Here is the ordering on $\s_0(R)$ that gives  the chains discussed in the introduction.

\begin{defn} \label{d0204'}
Given two classes $\cl B,\cl C\in\s_0(R)$, we write
$\cl B\tri\cl C$ when $C$ is totally $B$-reflexive, that is, when $C$ has finite
$\text{G}_B$-dimension; see Fact~\ref{f0217'}.
We write
$\cl B\trin\cl C$ when $\cl B\tri\cl C$ and $\cl B\neq\cl C$.
\end{defn}

The next facts contain some fundamental properties of this ordering.

\begin{fact} \label{f0216'}
Let $C$ be a semidualizing $R$-module.
Fact~\ref{f0211'} implies that $\cl C\tri\cl R$ and,
if $D$ is a dualizing $R$-module, then $\cl D\tri\cl C$.

Fact~\ref{f0210'} says that 
$\cl B\tri\cl C$ in $\s_0(R)$ if and only if $\cl{\bhat}\tri\cl{\chat}$ in $\s_0(\rhat)$;
also 
$\cl B\trin\cl C$ in $\s_0(R)$ if and only if $\cl{\bhat}\trin\cl{\chat}$ in $\s_0(\rhat)$
by Fact~\ref{f0204'}.
In other words, the injection $\s_0(R)\hookrightarrow\s_0(\rhat)$
perfectly respects the orderings on these two sets.
Compare this with Fact~\ref{f0216}.
\end{fact}

\begin{fact} \label{f0218'}
Let $B$ and $C$ be semidualizing $R$-modules such that $C$ is 
totally $B$-reflexive,
that is, such that $\cl B\tri\cl C$. By definition, this implies that 
$\ext^i_R(C,B)=0$ for all $i\geq 1$. In addition,
the $R$-module $\hom_R(C,B)$ is
semidualizing and totally $B$-reflexive; 
see~\cite[(2.11)]{christensen:scatac}.
Compare this with Fact~\ref{f0218}.
\end{fact}

Here is the key to the proofs of our main results when $R$ is Cohen-Macaulay.

\begin{fact} \label{f0220'}
Consider a chain 
$\cl{C^0}\tri\cl{C^1}\tri\cdots\tri\cl{C^d}$
in $\s_0(R)$. Gerko~\cite[(3.3)]{gerko:sdc} shows that there is an isomorphism
$$C^0\cong\hom_R(C^1,C^0)\otimes_R\cdots\otimes_R\hom_R(C^d,C^{d-1})\otimes_RC^d.$$
(Note that Fact~\ref{f0218'} implies that each factor in the tensor product is a semidualizing
$R$-module.)
The proof is by induction on $d$, with the case $d=1$ being the most important:
The natural evaluation homomorphism
$\xi\colon\hom_R(C^1,C^0)\otimes_RC^1\to C^0$
given by $\phi\otimes x\mapsto \phi(x)$
fits into the following commutative diagram
$$\xymatrix{
R\ar[rr]^-{\smash{\ol\chi}^R_{\hom_R(C^1,C^0)}}_-{\cong} \ar[d]_{\smash{\ol\chi}^R_{C^0}}^{\cong}
&&\hom_R(\hom_R(C^1,C^0),\hom_R(C^1,C^0)) \\
\hom_R(C^0,C^0)\ar[rr]^-{\hom_R(\xi,C^0)}
&&\hom_R(\hom_R(C^1,C^0)\otimes_RC^1,C^0). \ar[u]_{\cong}
}$$
The unspecified isomorphism is Hom-tensor adjointness.
The homomorphisms 
$\smash{\ol\chi}^R_{C^0}$ and $\smash{\ol\chi}^R_{\hom_R(C^1,C^0)}$ are isomorphisms
because $C^0$ and $\hom_R(C^1,C^0)$ are semidualizing;
see Fact~\ref{f0218'}. 
Hence, the homomorphism $\hom_R(\xi,C^0)$ is an isomorphism.
Since $C^0$ is semidualizing, it follows that
$\xi$ is an isomorphism; see~\cite[(A.8.11)]{christensen:gd}. 

Moreover, if $F^i$ is a free resolution of $\hom_R(C^i,C^{i-1})$ for 
$i=1,\ldots,d$ and $F^{d+1}$ is a projective resolution of $C^d$, then the tensor
product from Definition~\ref{d0105}
$$F^1\otimes_R\cdots\otimes_RF^d\otimes_RF^{d+1}$$ 
is a 
free resolution of $C^0$.
Compare this with Fact~\ref{f0220}.
\end{fact}

The final fact of this section demonstrates the utility of~\ref{f0220'}.
It compares to~\ref{f1001}.

\begin{fact} \label{f1001'}
The ordering on $\s_0(R)$ is reflexive 
by Fact~\ref{f0208'}.
Also, it is antisymmetric by~\cite[(5.3)]{takahashi:hiatsb}. 
The essential point in the proof of antisymmetry comes from Fact~\ref{f0220'}.
Indeed, if $\cl B\tri\cl C\tri\cl B$, then
$$B\cong \Hom_R(C,B)\otimes_R\Hom_R(B,C)\otimes_RB.$$
It follows that there is an equality of Betti numbers
$$\beta^R_0(B)=\beta^R_0(\Hom_R(C,B))\beta^R_0(\Hom_R(B,C))\beta^R_0(B)$$
and so $\Hom_R(C,B)$ and $\Hom_R(B,C)$ are cyclic. Fact~\ref{f0218'}
implies that $\Hom_R(C,B)$ is semidualizing, so we have $\Hom_R(C,B)\cong R$
by Fact~\ref{f0201'}.  This yields the second isomorphism in the next sequence
$$C\cong\hom_R(\Hom_R(C,B),B)\cong \hom_R(R,B)\cong B.$$
The first isomorphism follows from the fact that $C$ is totally $B$-reflexive,
and the third isomorphism is standard. We conclude that $\cl C=\cl B$.
\end{fact}

The question of 
transitivity for this relation
is another open question in this area.
It is open, even for artinian rings containing  a field.
Compare this to Question~\ref{q0202}.

\begin{question} \label{q0202'}
Let $A$, $B$ and $C$ be semidualizing $R$-modules.
If $B$ is totally $A$-reflexive and $C$ is totally $B$-reflexive, must $C$ be 
totally $A$-reflexive?
\end{question}

\section{Semidualizing Complexes} \label{sec02}

This section contains definitions and background material on semidualizing
complexes. In a sense, these are derived-category versions of the semidualizing
modules from the previous section.
(For notation and background information on the derived category $\catd(R)$,
consult Appendix~\ref{sec01}.)
Motivation also comes from Grothendieck's notion of a dualizing complex~\cite{hartshorne:rad} and
Avramov and Foxby's notion of a relative dualizing complex~\cite{avramov:rhafgd}.
The general definition  is due to Christensen~\cite{christensen:scatac}.

\begin{defn} \label{d0201}
Let $C$ be an $R$-complex. 
The \emph{homothety morphism} associated to $C$
in the category of $R$-complexes $\catc(R)$ 
is the morphism
$\smash{\ol\chi}^R_C\colon R\to\hom_R(C,C)$ given by
$\smash{\ol\chi}^R_C(r)(c)=rc$. 
This induces a well-defined
\emph{homothety morphism} associated to $C$
in $\catd(R)$ which is denoted $\chi^R_C\colon R\to\rhom_R(C,C)$.

The $R$-complex $C$ is \emph{semidualizing} if
it is homologically finite,
and the homothety morphism $\chi^R_C\colon R\to\rhom_R(C,C)$ is an isomorphism
in $\catd(R)$. 
An $R$-complex $D$ is \emph{dualizing} if it is semidualizing and
has finite injective dimension.
\end{defn}

The first fact of this section describes this
definition in terms of resolutions.

\begin{fact} \label{f0202}
Let $C$ be an $R$-complex. 
The morphism
$\chi^R_C\colon R\to\rhom_R(C,C)$ in $\catd(R)$
can be described  using a free resolution $F$ of $C$,
in which case it is represented by the morphism
$\smash{\ol\chi}^R_F\colon R\to\hom_R(F,F)$ in $\catc(R)$.
It can also be described  using an injective resolution $I$ of $C$,
in which case it is represented by 
$\smash{\ol\chi}^R_I\colon R\to\hom_R(I,I)$.
Compare this with~\cite[(2.1.2)]{christensen:gd}.
As this suggests, the semidualizing
property can be detected by any free (or injective) resolution
of $C$; and, when $C$ is semidualizing, the semidualizing property
is embodied by every free resolution and every injective resolution.
Here is the essence of the argument of one aspect of this statement;
the others are similar. The resolutions $F$ and $I$
are connected by a quasiisomorphism $\alpha\colon F\res I$ which yields 
the next commutative
diagram in $\catc(R)$
$$\xymatrix{
R\ar[rr]^-{\smash{\ol\chi}^R_F}\ar[d]_{\smash{\ol\chi}^R_I} 
&&\hom_R(F,F)\ar[d]^{\hom_R(F,\alpha)}_{\simeq} \\
\hom_R(I,I)\ar[rr]^-{\hom_R(\alpha,I)}_-{\simeq}
&&\hom_R(F,I).}$$
Hence, $\smash{\ol\chi}^R_F$ is a quasiisomorphism if and only if $\smash{\ol\chi}^R_I$ is
a quasiisomorphism. 
\end{fact}

The next fact compares Definitions~\ref{d0201'} and~\ref{d0201}.

\begin{fact} \label{f0222}
An $R$-module $C$ is semidualizing 
as an $R$-module
if and only if it is
semidualizing as an $R$-complex.
To see this, let $F$ be a free resolution of $M$.
The condition $\ext^i_R(C,C)=0$ is equivalent to 
the condition $\HH_{-i}(\hom_R(F,F))=0$ 
because of the following isomorphisms:
$$\ext^i_R(C,C)\cong\HH_{-i}(\rhom_R(C,C))\cong\HH_{-i}(\hom_R(F,F)).$$
(See, e.g., Fact~\ref{f0106}.)
Thus, we assume that
$\ext^i_R(C,C)=0$ for all $i\geq 1$.
In particular, 
since $\HH_i(R)=0$ for all $i\neq 0$,
 the map
$\HH_i(\chi^R_C)\colon \HH_i(R)\to\HH_i(\rhom_R(C,C))$ is an isomorphism
for all $i\neq 0$.
Next, there is a commutative diagram of $R$-module homomorphisms
where the unspecified isomorphisms are from Facts~\ref{f0103}
and~\ref{f0111}
$$\xymatrix{
R\ar[rr]^-{\smash{\ol\chi}^R_C}\ar[d]_{\cong}
&&\hom_R(C,C) \ar[d]^{\cong} \\
\HH_0(R) \ar[rr]^-{\HH_0(\chi^R_C)}
&&\HH_0(\rhom_R(C,C)).
}$$
It follows that $\smash{\ol\chi}^R_C$ is an isomorphism if and only if
$\HH_0(\chi^R_C)$ is an isomorphism, that is, if and only if
$\chi^R_C$ is an isomorphism in $\catd(R)$.
\end{fact}

The next fact documents the interplay between the semidualizing
property and the suspension operator.

\begin{fact} \label{f0203}
It is straightforward to show that an $R$-complex $C$ is
semidualizing if and only if some (equivalently, every) shift $\shift^iC$
is semidualizing; see~\cite[(2.4)]{christensen:scatac}. 
The essential point of the proof is that Fact~\ref{f4112} yields natural isomorphisms
$$\rhom_R(\shift^iC,\shift^iC)\simeq\shift^{i-i}\rhom_R(C,C)\simeq\rhom_R(C,C)$$
that are compatible with the homothety morphisms
$\chi^R_C$ and $\chi^R_{\shift^iC}$.
The analogous statement for dualizing complexes
follows from this because of the equality
$\id_R(\shift^iC)=\id_R(C)-i$ from Fact~\ref{f0101'}.
\end{fact}

\begin{disc} \label{d9998}
As in Remark~\ref{d9999}, we pause to explain some of the 
issues that arise when investigating semidualizing complexes in the non-local setting.
Let $S$ be a commutative noetherian ring, 
not necessarily local, and let $C$ be an $S$-complex.
Define the homothety homomorphism $\smash{\ol\chi}^S_C\colon S\to\hom_S(C,C)$,
the semidualizing property, and the set $\s(S)$ 
as in~\ref{d0201}. 

When $\spec(S)$ is connected, the set $\s(S)$ behaves similarly to $\s_0(S)$:
a nontrivial Picard group makes the ordering on $\s(S)$ non-antisymmetric,
and one can  overcome this by looking at an appropriate set of orbits.

However, when $\spec(S)$ is disconnected (that is, when  $S\cong S_1\times S_2$
for (nonzero) commutative noetherian rings $S_1$ and $S_2$) things are 
even more complicated than in the module-setting.
Indeed, the semidualizing $S$-complexes are all of the form
$\shift^iC_1\oplus \shift^jC_2$ where each $C_i$ is a semidualizing $S_i$-complex.
In other words, each connected component of $\spec(S)$ contributes essentially
two degrees of freedom to the elements of $\s(S)$.
For further discussion, see~\cite{avramov:rrc1,frankild:rrhffd,frankild:rbsc}.
\end{disc}

The next two facts are versions of~\ref{f0231'} and~\ref{f0215'} for semidualizing complexes.

\begin{fact} \label{f0201}
If $R$ is Gorenstein, then every semidualizing $R$-complex $C$ is isomorphic in $\catd(R)$
to $\shift^iR$ for some integer $i$ by~\cite[(8.6)]{christensen:scatac};
see also Theorem~\ref{prop0101}. 
(Note that the assumption that $R$ is local is crucial here
because of Remark~\ref{d9998}.) 
If $R$ is Cohen-Macaulay, then every semidualizing $R$-complex $C$ is isomorphic in $\catd(R)$
to $\shift^iB$ for some integer $i$ and some semidualizing $R$-module $B$
by~\cite[(3.4)]{christensen:scatac}. 
(In each case, we have $i=\inf(C)$. In the second case, we have
$B\cong\HH_i(C)$; see Facts~\ref{f0110} and~\ref{f0107}.
Again, this hinges on the assumption that $R$ is local.)
The converses of these statements hold when
$R$ has a dualizing complex
by~\cite[(8.6)]{christensen:scatac} and Fact~\ref{f0215'};  
the converses can fail when $R$ does not have a dualizing
complex; see~\cite[(5.5)]{christensen:dvke}. 
\end{fact}

\begin{fact} \label{f0215}
Grothendieck and Hartshorne~\cite[(V.10)]{hartshorne:rad}
and Kawasaki~\cite[(1.4)]{kawasaki}
show that $R$ has a dualizing complex if and only if $R$
is a homomorphic image of a Gorenstein ring.
In particular, if $R$ is complete, then Cohen's structure theorem
implies that $R$ has a dualizing complex.
\end{fact}

The next fact generalizes~\ref{f0206'}.

\begin{fact} \label{f0206}
Assume for this paragraph that $R$ has a dualizing complex $D$. Then 
there is a coefficientwise equality
$I^R_R(t)=t^sP^R_D(t)$ where $s=\dim(R)-\sup(D)$;
that is,
for all $i\in\bbz$ we have
$\mu^i_R(R)=\beta_{i-s}^R(D)$;
see, for instance,  \cite[(1.5.3),(2.6)]{avramov:rhafgd} and~\cite[(V.3.4)]{hartshorne:rad}.
Also, we have
$\sup(D)-\inf(D)=\dim(R)-\depth(R)$, that is, the range of nonvanishing homology of 
$D$ is the same as the Cohen-Macaulay defect of $R$; see~\cite[(3.5)]{christensen:scatac}.

More generally, let $D'$ be a dualizing complex for $\rhat$.
Then we have
$$I^R_R(t)=I^{\rhat}_{\rhat}(t)=t^sP^{\rhat}_{D'}(t)$$ 
where $s=\dim(\widehat R)-\sup(D')$;
in other words,
for all $i\in\bbz$ we have
$\mu^i_R(R)=\beta_{i-s}^{\rhat}(D')$.
Furthermore, we have
$$\sup(D')-\inf(D')=\dim(\rhat)-\depth(\rhat)=\dim(R)-\depth(R).$$
Compare this with Fact~\ref{f0206'}.
\end{fact}

Fact~\ref{f0201'} implies that a cyclic semidualizing $R$-module
must be isomorphic to the ring $R$.
Using the previous fact, we show next that a version of this statement 
for semidualizing complexes fails in general.
Specifically, there exists a ring $R$ that has a semidualizing $R$-complex $C$
that is not shift-isomorphic to $R$ even though its first nonzero Betti number is 1.
See Example~\ref{ex0601} for more on this ring.

\begin{ex} \label{ex0101}
Let $k$ be a field and set $R=k[\![X,Y]\!]/(X^2,XY)$. Then $R$ is a complete local ring
of dimension 1 and depth 0. 
Hence $R$ has a dualizing complex $D$. 
Apply a shift if necessary to assume that $\inf(D)=0$.
Then Fact~\ref{f0206} provides the first equality in the next sequence
$$P^R_D(t)=I^R_R(t)=1+2t+2t^2+\cdots$$
while the second equality is from, e.g., \cite[Ex.\ 1]{christensen:iabbn}.
In particular, we have $\beta^R_0(D)=1$
and $\beta^R_i(D)=0$ for all $i<0$ even though $D\not\simeq\shift^jR$
for all $j\in\bbz$.
\end{ex}

We shall use the next definition  to equate semidualizing complexes that are
essentially the same. This compares with the identification of isomorphic modules
in Definition~\ref{d0201}; see
Fact~\ref{f0214}.

\begin{defn} \label{d0203}
Given two $R$-complexes $B$ and $C$,  if there is an integer $i$ such that
$C\simeq\shift^iB$, then 
$B$ and $C$ are
\emph{shift-isomorphic}.\footnote{This yields an equivalence relation on the class of all semidualizing
$R$-complexes:
(1) One has $C\simeq \shift^0C$;
(2) If $C\simeq\shift^iB$, then $B\simeq\shift^{-i}C$;
(3) If $C\simeq\shift^iB$ and $B\simeq\shift^jA$, then $C\simeq\shift^{i+j}A$.}
The set of ``shift-isomorphism classes'' of semidualizing $R$-complexes is denoted
$\s(R)$, and the shift-isomorphism class of a semidualizing $R$-complex $C$ is
denoted $\cl C$. 
\end{defn}

The next fact compares Definitions~\ref{d0201'} and~\ref{d0203}.

\begin{fact} \label{f0214}
It is straightforward to show that
the natural embedding of $\catm(R)$ inside $\catd(R)$ induces a natural
injection $\s_0(R)\hookrightarrow\s(R)$; see Facts~\ref{f0222} and~\ref{f0103}. 
This injection is surjective when $R$ is Cohen-Macaulay by Fact~\ref{f0201}.
(Note that the assumption that $R$ is local is essential here
because of Remark~\ref{d9998}.)
\end{fact}

Here is the version of Question~\ref{q0201'} for semidualizing complexes.
Again, Remark~\ref{d9998} shows that the assumption that $R$ is local is crucial.
Fact~\ref{f0214} shows that an affirmative answer to 
Question~\ref{q0201} would yield an affirmative answer to 
Question~\ref{q0201'}. Also note that
the methods of this paper
do not answer this question, even though Theorem~\ref{prop0103} shows that
$\s(R)$ cannot  have arbitrarily long chains.

\begin{question} \label{q0201}
Is the set  $\s(R)$ finite?
\end{question}

The next properties compare to those in Fact~\ref{f0204'}.

\begin{fact} \label{f0204}
When $C$ is a homologically finite $R$-complex, it is semidualizing for $R$
if and only if the base-changed complex $\rhat\lotimes_RC$ is semidualizing for $\rhat$.
The essential point of the proof is that Fact~\ref{f4112} provides
the following isomorphism in $\catd(\rhat)$
$$\rhom_{\rhat}(\rhat\lotimes_RC,\rhat\lotimes_RC)
\simeq\rhat\lotimes_R\rhom_R(C,C)$$
which is compatible with the corresponding homothety morphisms.
The parallel statement for dualizing objects
also holds;
see~\cite[(5.6)]{christensen:scatac}
and~\cite[(V.3.5)]{hartshorne:rad}.

Given two homologically finite $R$-complexes $B$ and $C$, we have
$C\simeq \shift^iB$ if and only if $\rhat\lotimes_RC\simeq\rhat\lotimes_R\shift^iB$
by~\cite[(1.11)]{frankild:rrhffd}. Combining this with the previous paragraph,
we see that the assignment $C\mapsto\rhat\lotimes_RC$ induces a well-defined
injective function $\s(R)\hookrightarrow\s(\rhat)$. The restriction to $\s_0(R)$
is precisely the induced map from Fact~\ref{f0204'},
and thus there is a commutative diagram
$$\xymatrix{
\s_0(R) \ar@{^(->}[r]\ar@{^(->}[d] &\s_0(\rhat) \ar@{^(->}[d] \\
\s(R) \ar@{^(->}[r] &\s(\rhat). 
}$$
\end{fact}

The following fact compares to~\ref{f0205'}; see also Lemma~\ref{l0201}
and Question~\ref{q}.

\begin{fact} \label{f0205}
If $C$ is a semidualizing $R$-complex and $\pd_R(C)<\infty$, then $C\simeq\shift^iR$
where $i=\inf(C)$ by~\cite[(8.1)]{christensen:scatac}. 
(As in Fact~\ref{f0205'}, this relies on the local assumption on $R$.)
\end{fact}

Here is a version of Definition~\ref{d0202'} for semidualizing complexes.
It originates with the special cases of ``reflexive complexes'' from~\cite{hartshorne:rad,yassemi:gd}.
The definition in this generality is from~\cite{christensen:scatac}.

\begin{defn} \label{d0202}
Let $C$  and $X$ be $R$-complexes. 
The \emph{biduality morphism} associated to $C$ and $X$ in $\catc(R)$
is the morphism
$\smash{\ol\delta}^C_X\colon X\to\hom_R(\hom_R(X,C),C)$
given by $((\smash{\ol\delta}^C_X)_p(x))_q(\{\phi_j\}_{j\in\bbz})=(-1)^{pq}\phi_p(x)$. 
This yields a well-defined \emph{biduality morphism} 
$\delta^C_X\colon X\to\rhom_R(\rhom_R(X,C),C)$
associated to $C$ and $X$ 
in $\catd(R)$.

Assume that $C$ is a semidualizing $R$-complex.
The $R$-complex $X$ is
\emph{$C$-reflexive} when it satisfies the following properties:
\begin{enumerate}[\quad(1)]
\item \label{d0202a}
The complex $X$ is homologically finite;
\item \label{d0202c}
The biduality morphism
$\delta^C_X\colon X\to\rhom_R(\rhom_R(X,C),C)$ in $\catd(R)$ is an 
isomorphism;
and
\item \label{d0202b}
The complex $\rhom_R(X,C)$ is homologically 
bounded, i.e., finite.
\end{enumerate}
\end{defn}

\begin{disc}
When $C$ is a semidualizing $R$-complex, every homologically finite
$R$-complex $X$ has a well-defined $\text{G}_C$-dimension which is finite
precisely when $X$ is  nonzero and $C$-reflexive. 
(Note that this invariant is not described in terms of resolutions.)
We shall not need this invariant
here; the interested reader should consult~\cite{christensen:scatac}.
\end{disc}

\begin{disc}
Avramov and Iyengar~\cite[(1.5)]{avramov:gahc} have shown 
that condition~\eqref{d0202b}
of Definition~\ref{d0202} is redundant when $C=R$. The same proof shows that
this condition is redundant in general. However, the proof of this fact is outside
the scope of the present article, so we continue to state this condition explicitly.
\end{disc}

The next fact shows that, as with the semidualizing property, the 
reflexivity property is independent of the choice of resolutions.

\begin{fact} \label{f0207}
Let $C$ and $X$ be  $R$-complexes and assume that $C$ is semidualizing. 
The biduality morphism
$\delta^C_X\colon X\to\rhom_R(\rhom_R(X,C),C)$ in $\catd(R)$
can be described  using an injective resolution $I$ of $C$,
in which case it is represented by the morphism
$\smash{\ol\delta}^I_X\colon X\to\hom_R(\hom_R(X,I),I)$. 
Compare this with~\cite[(2.1.4)]{christensen:gd}.
As with the semidualizing
property, reflexivity can be detected by any injective resolution
of $C$; and, when $X$ is $C$-reflexive, the reflexivity
is embodied by every  injective resolution.
Here is the essence of the argument. Let $I$ and $J$ be injective resolutions
of $C$. 
Fact~\ref{f0106} implies that
$$\hom_R(X,I)\simeq\rhom_R(X,C)\simeq\hom_R(X,J)$$
and so $\hom_R(X,I)$ is homologically bounded if and only if
$\hom_R(X,J)$ is homologically bounded.
Furthermore, there is a quasiisomorphism $\alpha\colon I\res J$, and this yields the next
commutative diagram in $\catc(R)$
$$\xymatrix{
X \ar[rrr]^-{\smash{\ol\delta}^I_X}\ar[d]_-{\smash{\ol\delta}^J_X}
&&&\hom_R(\hom_R(X,I),I) \ar[d]_{\simeq}^{\hom_R(\hom_R(X,I),\alpha)}\\
\hom_R(\hom_R(X,J),J)\ar[rrr]^{\hom_R(\hom_R(X, \alpha),J)}_-{\simeq}
&&& \hom_R(\hom_R(X,I),J).
}$$
Hence $\smash{\ol\delta}^I_X$ is a quasiisomorphism if and only if
$\smash{\ol\delta}^J_X$ is a quasiisomorphism.
\end{fact}

We next compare Definition~\ref{d0202} with the corresponding notions from
Section~\ref{sec12}.

\begin{fact} \label{f0208}
Let $C$ be a semidualizing $R$-module,
and let $G$ be a finitely generated $R$-module.
If  $G$ is totally $C$-reflexive, then it
is $C$-reflexive as a complex. 
Indeed, the following isomorphisms
imply  that $\rhom_R(G,C)$ is homologically bounded.
$$\HH_{i}(\rhom_R(G,C))
\cong\ext^{-i}(G,C)
\cong\begin{cases} 0 & \text{if $i\neq 0$} \\ \hom_R(G,C) & \text{if $i=0$.}\end{cases}$$
(See Fact~\ref{f0106}.)
Fact~\ref{f0110} explains the first isomorphism in the next sequence
\begin{align*}
\rhom_R(G,C)\simeq\HH_0(\rhom_R(G,C))&\cong\ext^0_R(G,C)\cong\hom_R(G,C)\\
\HH_i(\rhom_R(\rhom_R(G,C),C))
&\cong\HH_i(\rhom_R(\hom_R(G,C),C))\\
&\cong\ext^{-i}_R(\hom_R(G,C),C)\\
&\cong\begin{cases} 0 & \text{if $i\neq 0$} \\ G & \text{if $i=0$}\end{cases}
\end{align*}
and the others follow from the previous display and Fact~\ref{f0106}.
Thus, for all $i\neq 0$, the function
$\HH_i(\delta^C_G)\colon\HH_i(G)\to\HH_i(\rhom_R(\rhom_R(G,C),C))$
maps from 0 to 0 and thus
is an isomorphism.
To show that $G$ is $C$-reflexive, it remains to show that the map
$\HH_0(\delta^C_G)$
is  an isomorphism.
Check that there is a commutative diagram
$$\xymatrix{
G \ar[rr]^-{\smash{\ol\delta}^C_G}_-{\cong}\ar[d]_{\cong}
&&\hom_R(\hom_R(G,C),C) \ar[d]^{\cong} \\
\HH_0(G) \ar[rr]^-{\HH_0(\delta^C_G)}
&&\HH_0(\rhom_R(\rhom_R(G,C),C))
}$$
where the unspecified isomorphisms are essentially from Fact~\ref{f0103}.
Thus $\HH_0(\delta^C_G)$
is  an isomorphism as desired.

More generally, a finitely generated $R$-module has finite
$\text{G}_C$-dimension if and only if it is $C$-reflexive as an $R$-complex.
(Thus, the converse of the second sentence of the previous paragraph fails in general.)
Furthermore, a homologically finite
$R$-complex $X$ is $C$-reflexive if and only if there is an isomorphism
$X\simeq H$ in $\catd(R)$ where $H$ is a bounded complex of totally
$C$-reflexive $R$-modules.
See~\cite[(3.1)]{holm:smarghd}. 
\end{fact}

The next fact includes  versions of~\ref{f0208'} and~\ref{f0211'} for semidualizing complexes.

\begin{fact} \label{f0299}
Let $C$ be a semidualizing $R$-complex. Every finitely generated free
$R$-module is $C$-reflexive, as is $C$ itself.
The  essential point of the proof is that the following isomorphisms
are compatible with the corresponding biduality morphisms:
\begin{align*}
\rhom_R(\rhom_R(R,C),C)
&\simeq\rhom_R(C,C)\simeq R \\
\rhom_R(\rhom_R(C,C),C)
&\simeq\rhom_R(R,C)\simeq C.
\end{align*}
See~\cite[(2.8)]{christensen:scatac} and Fact~\ref{f4112}.

If $X$ is a homologically
finite $R$-complex of finite projective dimension, then $X$ is $C$-reflexive
by~\cite[(2.9)]{christensen:scatac}.
If $D$ is a dualizing $R$-complex, 
then every homologically finite $R$-complex is $D$-reflexive.
Conversely, if the residue field
$k$ is $C$-reflexive, then $C$ is dualizing.
See~\cite[(8.4)]{christensen:scatac} or~\cite[(V.2.1)]{hartshorne:rad}.
\end{fact}

As with the semidualizing property, reflexivity is
independent of shift.

\begin{fact} \label{f0209}
Let $C$ be a semidualizing $R$-complex.
It is straightforward to show that an $R$-complex $X$ is
$C$-reflexive  if and only if some (equivalently, every) shift $\shift^iX$
is $\shift^jC$-reflexive for some (equivalently, every) integer $j$. 
The point is that Fact~\ref{f4112} yields natural isomorphisms
\begin{align*}
\rhom_R(\shift^iX,\shift^jC)
&\simeq\shift^{j-i}\rhom_R(X,C) \\
\rhom_R(\rhom_R(\shift^iX,\shift^jC),\shift^jC)
&\simeq\rhom_R(\shift^{j-i}\rhom_R(X,C),\shift^jC) \\
&\simeq\shift^{j-(j-i)}\rhom_R(\rhom_R(X,C),C) \\
&\simeq\shift^{i}\rhom_R(\rhom_R(X,C),C) 
\end{align*}
that are compatible with $\delta^C_X$ and $\delta^{\shift^jC}_{\shift^iX}$.
\end{fact}

The next fact is a version of~\ref{f0210'} for semidualizing complexes.

\begin{fact} \label{f0210}
If $C$ is a semidualizing $R$-complex, then
a given homologically finite $R$-complex $X$ is $C$-reflexive if and only if
the base-changed complex $\rhat\lotimes_RX$ is $\rhat\lotimes_RC$-reflexive;
see~\cite[(5.10)]{christensen:scatac}.
The main point of the proof  is that 
Fact~\ref{f4112} provides the following isomorphisms in $\catd(\rhat)$ 
\begin{align*}
\rhom_{\rhat}(\rhat\lotimes_RX,\rhat\lotimes_RC)
&\simeq\rhat\lotimes_R\rhom_R(X,C) \\
\rhom_{\rhat}(\rhom_{\rhat}(\rhat\lotimes_RX,\rhat\lotimes_RC),\rhat\lotimes_RC)
\hspace{-2cm}\\
&\simeq\rhom_{\rhat}(\rhat\lotimes_R\rhom_R(X,C),\rhat\lotimes_RC) \\
&\simeq\rhat\lotimes_R\rhom_R(\rhom_R(X,C),C)
\end{align*}
and  that these isomorphisms are compatible with
$\delta^C_X$ and $\delta^{\rhat\lotimes_RC}_{\rhat\lotimes_RX}$.
\end{fact}

Here is the ordering on $\s(R)$ used in our main results.

\begin{defn} \label{d0204}
Given two classes $\cl B,\cl C\in\s(R)$, we write
$\cl B\tri\cl C$ when $C$ is $B$-reflexive;
we write
$\cl B\trin\cl C$ when $\cl B\tri\cl C$ and $\cl B\neq\cl C$.
\end{defn}

The following fact compares this relation with the one from Definition~\ref{d0204'}.

\begin{fact} \label{f0298}
Combining Fact~\ref{f0217'} and the last paragraph of Fact~\ref{f0208}, we see that,
if $B$ and $C$ are semidualizing $R$-modules, then
$\cl B\tri\cl C$ in $\s(R)$ if and only if $\cl B\tri\cl C$ in $\s_0(R)$,
and $\cl B\trin\cl C$ in $\s(R)$ if and only if $\cl B\trin\cl C$ in $\s_0(R)$
That is, the map $\s_0(R)\hookrightarrow\s(R)$ perfectly
respects the orderings on these two sets.
\end{fact}

The next facts compare with~\ref{f0216'} and~\ref{f0218'}.

\begin{fact} \label{f0216}
Let $C$ be a semidualizing $R$-complex.
Fact~\ref{f0299} implies that $\cl C\tri\cl R$, and
if $D$ is a dualizing $R$-complex, then $\cl D\tri\cl C$.

Fact~\ref{f0210} says that 
$\cl B\tri\cl C$ in $\s(R)$ if and only if $\cl{\rhat\lotimes_RB}\tri\cl{\rhat\lotimes_RC}$ in $\s(\rhat)$;
also 
$\cl B\trin\cl C$ in $\s(R)$ if and only if $\cl{\rhat\lotimes_RB}\trin\cl{\rhat\lotimes_RC}$ in $\s(\rhat)$
by Fact~\ref{f0204}.
So, the injection $\s(R)\hookrightarrow\s(\rhat)$ perfectly 
respects the orderings on these two sets.
\end{fact}

\begin{fact} \label{f0218}
Let $B$ and $C$ be semidualizing $R$-complexes such that $C$ is $B$-reflexive,
that is, such that $\cl B\tri\cl C$. This implies that the complex
$\rhom_R(C,B)$ is homologically finite, by definition. Moreover~\cite[(2.11)]{christensen:scatac} 
shows that $\rhom_R(C,B)$ is
semidualizing and $B$-reflexive.
The main point of the proof is that there is a
sequence of isomorphisms
\begin{align*}
\rhom_R(\rhom_R(C,B),\rhom_R(C,B)) \hspace{-2cm} \\
&\simeq\rhom_R(\rhom_R(C,B)\lotimes_RC,B) \\
&\simeq\rhom_R(C,\rhom_R(\rhom_R(C,B),B)) \\
&\simeq\rhom_R(C,C) \\
&\simeq R.
\end{align*}
The first two isomorphisms are  Hom-tensor adjointness~\ref{f4112}.
The third isomorphism is from the assumption that $C$ is $B$-reflexive,
and the fourth isomorphism is from the fact that $C$ is semidualizing.
\end{fact}

The next fact  compares to~\ref{f0220'}. It is the key
tool for our main results.

\begin{fact} \label{f0220}
Consider a chain 
$\cl{C^0}\tri\cl{C^1}\tri\cdots\tri\cl{C^d}$
in $\s(R)$. Gerko~\cite[(3.3)]{gerko:sdc} shows that there is an isomorphism
$$C^0\simeq\rhom_R(C^1,C^0)\lotimes_R\cdots\lotimes_R\rhom_R(C^d,C^{d-1})\lotimes_RC^d.$$
(Note that each factor in the tensor product is a semidualizing
$R$-complex by Fact~\ref{f0218}.)
The proof is by induction on $d$, with the case $d=1$ being the most important.
Consider  the natural evaluation morphism
$$\xi\colon\rhom_R(C^1,C^0)\lotimes_RC^1\to C^0$$
which fits into the following commutative diagram:
$$\xymatrix{
R\ar[rr]^-{\chi^R_{\rhom_R(C^1,C^0)}}_-{\simeq} \ar[d]_{\chi^R_{C^0}}^{\simeq}
&&\rhom_R(\rhom_R(C^1,C^0),\rhom_R(C^1,C^0)) \\
\rhom_R(C^0,C^0)\ar[rr]^-{\rhom_R(\xi,C^0)}
&&\rhom_R(\rhom_R(C^1,C^0)\lotimes_RC^1,C^0). \ar[u]_{\simeq}
}$$
The unspecified isomorphism is adjointness~\ref{f4112}.
The morphisms $\chi^R_{\rhom_R(C^1,C^0)}$ and
$\chi^R_{C^0}$  are isomorphisms
in $\catd(R)$ since $C^0$ and $\rhom_R(C^1,C^0)$ are semidualizing;
see Fact~\ref{f0218}. 
Hence, the morphism $\rhom_R(\xi,C^0)$ is an isomorphism in $\catd(R)$.
Since $C^0$ is semidualizing, it follows that
$\xi$ is also an isomorphism; see~\cite[(A.8.11)]{christensen:gd}.
\end{fact}

The final fact in this section compares to~\ref{f1001'}.

\begin{fact} \label{f1001}
The ordering on $\s(R)$ is reflexive 
by Fact~\ref{f0208}.
Also, it is antisymmetric by~\cite[(5.3)]{takahashi:hiatsb}. 
The essential point in the proof of antisymmetry comes from Fact~\ref{f0220}.
Indeed, if $\cl B\tri\cl C\tri\cl B$, then
$$B\simeq \rhom_R(C,B)\lotimes_R\rhom_R(B,C)\lotimes_RB.$$
It follows that there is an equality of Poincar\'e series
$$P^R_B(t)=P^R_{\rhom_R(C,B)}(t)P^R_{\rhom_R(B,C)}(t)P^R_{B}(t).$$
Since each Poincar\'e series has nonnegative integer coefficients,
this display implies that 
$P^R_{\rhom_R(C,B)}(t)=t^r$ and $P^R_{\rhom_R(B,C)}(t)=t^{-r}$
for some integer $r$. So, we have $\rhom_R(C,B)\simeq \shift^rR$.
This yields the second isomorphism in the next sequence
$$C \simeq\rhom_R(\rhom_R(C,B),B) \simeq \rhom_R(\shift^rR,B) \simeq \shift^rB.$$
The first isomorphism follows from the fact that $C$ is $B$-reflexive,
and the third isomorphism is cancellation~\ref{f4112}. We conclude that $\cl C=\cl B$.
\end{fact}

As in the module-setting, the question of
the transitivity of this order remains  open.
An affirmative answer to Question~\ref{q0202}
would yield an affirmative answer to Question~\ref{q0202'} as
the map $\s_0(R)\hookrightarrow\s(R)$ is  order-preserving
by Fact~\ref{f0298}. 
Questions~\ref{q0202'} and~\ref{q0202} are equivalent when $R$ is Cohen-Macaulay
since, in this case, the map $\s_0(R)\hookrightarrow\s(R)$ is surjective
by Fact~\ref{f0214}. (Again, this hinges on the local assumption for $R$ by
Remark~\ref{d9998}.)

\begin{question} \label{q0202}
Let $A$, $B$ and $C$ be semidualizing $R$-complexes.
If $B$ is $A$-reflexive and $C$ is $B$-reflexive, must $C$ be $A$-reflexive?
\end{question}

\section{Bounding Bass Numbers} \label{sec05}

We begin with three lemmas, the first of which essentially
says that semidualizing complexes over local rings are indecomposable.
Note that Remark~\ref{d9998} shows that the local hypothesis is essential.

\begin{lem} \label{l0200}
Let $R$ be a local ring
and let $C$ be a semidualizing $R$-complex.
If $X$ and $Y$ are $R$-complexes such that
$C\simeq X\oplus Y$, then either $X\simeq 0$ or $Y\simeq 0$.
\end{lem}

\begin{proof}
The condition $C\simeq X\oplus Y$ implies that
$\HH_i(C)\cong\HH_i(X)\oplus\HH_i(Y)$ for each index $i$.
Hence, the fact that $C$ is homologically finite implies that
$X$ and $Y$ are both homologically finite as well.

We assume that $X\not\simeq 0$ and show that $Y\simeq 0$. 
Fact~\ref{f0102} yields the following equality of formal Laurent series
$$I_R^{\rhom_R(X,X)}(t)=P^R_X(t)I^X_R(t).$$
The condition $X\not\simeq 0$ implies
$P^R_X(t)\neq 0$ and $I^X_R(t)\neq 0$ by Fact~\ref{f0102}.
The display implies that $I_R^{\rhom_R(X,X)}(t)\neq 0$,
and thus $\rhom_R(X,X)\not\simeq 0$. 
The fact that $C$ is a semidualizing $R$-complex yields the first
isomorphism in the next sequence
\begin{align*}
R
&\simeq\rhom_R(C,C)
\simeq\rhom_R(X\oplus Y,X\oplus Y)\\
&\simeq\rhom_R(X,X) \oplus\rhom_R(X,Y) \oplus\rhom_R(Y,X) \oplus\rhom_R(Y,Y).
\end{align*}
The third isomorphism is additivity~\ref{f4112}.
Because $R$ is local, it is indecomposible as an $R$-module.
By taking homology, we conclude that three of the summands
in the second line of the previous sequence
are homologically trivial, that is $\simeq 0$.
Since $\rhom_R(X,X)\not\simeq 0$, it follows that
$\rhom_R(Y,Y)\simeq 0$.
Another application of Fact~\ref{f0102} implies that
$$0=I_R^{\rhom_R(Y,Y)}(t)=P^R_Y(t)I^Y_R(t).$$
Hence, either $P^R_Y(t)=0$ or $I^Y_R(t)=0$. 
In either case, we conclude that $Y\simeq 0$.
\end{proof}

The next lemma
generalizes Fact~\ref{f0205}. See also Fact~\ref{f0205'}
and Question~\ref{q}. It is essentially a corollary of Lemma~\ref{l0200}.

\begin{lem} \label{l0201}
Let $R$ be a local ring
and let $C$ be a semidualizing $R$-complex.
Set $i=\inf(C)$. If there  is an integer $j\geq i$ such that 
$\beta_j^R(C)=0$, then 
$C\simeq\shift^iR$.
\end{lem}

\begin{proof}
By Fact~\ref{f0205}, it suffices to show that $\pd_R(C)<\infty$. 
Let $F$ be a minimal free resolution of $C$. 
The assumption $\beta_j^R(C)=0$ implies that $F_j=0$
by Fact~\ref{f0102}.
Note that $F_i\neq 0$ since $H_i(C)\neq 0$, so we have $j>i$.
Thus $F$ has the following form
$$
F=\cdots\xra{\partial^F_{j+2}}F_{j+1}\to 0\to F_{j-1}
\xra{\partial^F_{j-1}}\cdots\xra{\partial^F_{i+1}}F_{i}\to 0.
$$
Hence, we have
$C\simeq F\cong F^1\oplus F^2$ where
\begin{align*}
F^1&=\cdots\xra{\hspace{8mm}} 0\xra{\hspace{4.5mm}} 0\to F_{j-1}
\xra{\partial^F_{j-1}}\cdots\xra{\partial^F_{i+1}}F_{i}\to 0 \\
F^2&=\cdots\xra{\partial^F_{j+2}}F_{j+1}\to 0\xra{\hspace{4mm}} 0
\xra{\hspace{8.5mm}}\cdots\xra{\hspace{6mm}} 0 \xra{\hspace{3mm}} 0.
\end{align*}
The condition $F_i\neq 0$ implies $F^1\not\simeq 0$ as $F^1$ is minimal;
see Fact~\ref{f0101}.
Lemma~\ref{l0200} yields $F^2\simeq 0$, 
so $C\simeq F^1\oplus F^2\simeq F^1$, which has finite projective dimension.
\end{proof}

When $R$ is Cohen-Macaulay, the gist of the proof of the next lemma is found in
Fact~\ref{f0220'}: the minimal free resolution of $D$ factors as a tensor product
of $d+1$ minimal free resolutions of modules of infinite projective dimension.
Note that the Cohen-Macaulay hypothesis in the final sentence of the statement
is essential because of Example~\ref{ex0101}.

\begin{lem} \label{l0501}
Let $R$ be a local ring of depth $g$
such that $\s(R)$ contains a chain of length $d+1$.
Then there exist power series $P_0(t),\ldots,P_d(t)$ with positive integer
coefficients such that
$$I^R_R(t)=t^{g}P_0(t)\cdots P_d(t).$$
If, in addition, $R$ is Cohen-Macaulay, and 
$p$ is the smallest prime factor of $\mu^g_R(R)$, then the constant term of each $P_i(t)$ is at least $p$.
\end{lem}

\begin{proof}
Assume that 
$\s(R)$ contains a chain 
$\cl{C^0}\trin\cl{C^1}\trin\cdots\trin\cl{C^d}\trin\cl{C^{d+1}}$.

We begin by proving the result in the case where $R$ has a dualizing complex $D$.
Applying a suspension if necessary, we assume that $\sup(D)=\dim(R)$;
see Fact~\ref{f0203}.
It follows that $\inf(D)=g$ by Fact~\ref{f0206}.
From Fact~\ref{f0206} we conclude that there is a formal equality of power series
$I^R_R(t)=P^R_D(t)$.
Fact~\ref{f0216} implies that $\cl D\tri\cl{C^0}$.
Hence, we may extend the given chain by adding the link
$\cl D\tri\cl{C^0}$ if necessary in order to assume 
that $C^0= D$. Similarly, we assume that $C^{d+1}=R$.

Fact~\ref{f0218} implies that, for $i=0,\ldots,d$ the  $R$-complex
$\rhom_R(C^{i+1},C^i)$ is semidualizing and $C^i$-reflexive. We observe
that $\cl{\rhom_R(C^{i+1},C^i)}\neq\cl R$.
Indeed, if not, then
$\rhom_R(C^{i+1},C^i)\simeq \shift^jR$ for some $j$, 
and this explains the second isomorphism in the following sequence.
$$C^{i+1}\simeq\rhom_R(\rhom_R(C^{i+1},C^i),C^i)
\simeq\rhom_R(\shift^jR,C^i)\simeq \shift^jC^i$$
The first isomorphism is by Definition~\ref{d0202}\eqref{d0202c},
and the third one is cancellation~\ref{f4112}.
These isomorphisms imply that $\cl{C^{i+1}}=\cl{C^{i}}$, contradicting our assumption
that $\cl{C^{i+1}}\trin\cl{C^{i}}$.

Set $m_i=\inf(\rhom_R(C^{i+1},C^i))$. Lemma~\ref{l0201} implies that 
$$\beta^R_j(\rhom_R(C^{i+1},C^i))\geq 1$$
for each $j\geq m_i$.
It follows that the series
$$P_i(t)=\sum_{n=0}^{\infty}\beta^R_{n+m_i}(\rhom_R(C^{i+1},C^i))t^n$$
is a power series with positive integer coefficients such that
\begin{equation} \label{thm0101f}
P^R_{\rhom_R(C^{i+1},C^i)}(t)=t^{m_i}P_i(t).
\end{equation}
Fact~\ref{f0220} yields the first isomorphism in the following sequence
\begin{equation} \label{thm0101a}
\begin{split}
D&= C^0 
\simeq\rhom_R(C^1,C^0)\lotimes_R\cdots\lotimes_R\rhom_R(C^{d+1},C^{d})\lotimes_RC^{d+1} \\
&\simeq\rhom_R(C^1,C^0)\lotimes_R\cdots\lotimes_R\rhom_R(C^{d+1},C^{d}).
\end{split}
\end{equation}
The equality and the second isomorphism are from the assumptions $C^0=D$ and $C^{d+1}=R$.
It follows from Fact~\ref{f0111} that
\begin{equation} \label{thm0101d}
g=\inf(D)=\sum_{i=0}^dm_i.
\end{equation}
The second equality in the next sequence follows from~\eqref{thm0101a}
using Fact~\ref{f0102}
\begin{align*}
I^R_R(t)
&=P^R_D(t) \\
&=P^R_{\rhom_R(C^{1},C^0)}(t)\cdots P^R_{\rhom_R(C^{d+1},C^d)}(t) \\
&=\left(t^{m_0}P_0(t)\right)\cdots \left(t^{m_d}P_d(t)\right) \\
&=t^{g}P_0(t)\cdots P_d(t)
\end{align*}
The first equality is by the choice of $D$;
the third equality is from~\eqref{thm0101f};
and the fourth equality is from~\eqref{thm0101d}.

Assume for this paragraph that $R$ is Cohen-Macaulay. Fact~\ref{f0201} yields
an isomorphism
$\rhom_R(C^{i+1},C^i)\simeq\shift^{s_i}B^i$ where 
$s_i=\inf(\rhom_R(C^{i+1},C^i))$ and $B^i$ is the semidualizing
$R$-module $\HH_{s_i}(\rhom_R(C^{i+1},C^i))$.
Since $\rhom_R(C^{i+1},C^i)$ is non-free, Fact~\ref{f0201'}
implies that $\beta^R_0(B^i)\geq 2$; this is the constant term of 
$P_i(t)$. The formula $I^R_R(t)=t^{g}P_0(t)\cdots P_d(t)$
implies that $\mu^g_R(R)$ is the product of the constant terms of the $P_i(t)$;
since each constant term is at least 2, it must be at least $p$.
This completes the proof in the case where $R$ has a dualizing complex.

Finally, we prove the result in general. The completion $\rhat$ has a dualizing
complex by Fact~\ref{f0215}. Also, the given chain gives rise to the following
chain in $\s(\rhat)$
$$\cl{\rhat \lotimes_RC^0}\trin\cl{\rhat \lotimes_RC^1}\trin\cdots\trin\cl{\rhat \lotimes_RC^d}\trin
\cl{\rhat \lotimes_RC^{d+1}}$$
by Fact~\ref{f0216}.
The previous case
yields power series $P_0(t),\ldots,P_d(t)$ with positive integer
coefficients such that
$I^{\rhat}_{\rhat}(t)=t^{\depth(\rhat)}P_0(t)\cdots P_d(t)$.
Hence, the desired conclusion follows from the equalities
$g=\depth(R)=\depth(\rhat)$ and
$I^R_R(t)=I^{\rhat}_{\rhat}(t)$, and the fact that
$R$ is Cohen-Macaulay if and only if $\rhat$ is Cohen-Macaulay.
\end{proof}

\begin{disc} \label{d1}
It is straightforward to use Fact~\ref{f0216} to give a slight strengthening of
Lemma~\ref{l0501}. Indeed, the condition ``$\s(R)$ contains a chain of length $d+1$''
is stronger than necessary; 
the proof shows that one can derive the same conclusions only assuming
that $\s(\rhat)$ contains a chain of length $d+1$.
Similar comments hold true for the remaining results in this section
and for the results of Section~\ref{sec06}.
\end{disc}

The next two results contain Theorem~\ref{thm0001} from the introduction
and follow almost directly from Lemma~\ref{l0501}.

\begin{thm} \label{thm0101}
Let $R$ be a local ring.
If $\s(R)$ contains a chain of length $d+1$,
then the sequence of Bass numbers $\{\mu^i_R(R)\}$ is bounded below by a polynomial in $i$
of degree~$d$.
\end{thm}

\begin{proof}
Assume that 
$\s(R)$ contains a chain 
of length $d+1$.
Lemma~\ref{l0501} implies that there exist power series $P_0(t),\ldots,P_d(t)$ with positive integer
coefficients satisfying the equality in the following sequence
\begin{align*}
I^R_R(t)
=t^{\depth(R)}P_0(t)\cdots P_d(t)
\succeq t^{\depth(R)}\left(\sum_{n=0}^\infty t^n\right)^{\!\!d+1}.
\end{align*}
The coefficientwise inequality follows from the fact that each coefficient of 
$P_j(t)$ is a positive integer.

It is well known that the degree-$i$ coefficient of the series
$\left(\sum_{n=0}^\infty t^n\right)^{d+1}$ is given by a
polynomial in $i$ of degree $d$. It follows that the same is true of the 
coefficients of the series
$t^{\depth(R)}\left(\sum_{n=0}^\infty t^n\right)^{d+1}$. Hence, the 
degree-$i$ coefficient of the Bass series $I^R_R(t)$,
i.e., the $i$th Bass number $\mu^i_R(R)$, is bounded below
by such a polynomial. 
\end{proof}

\begin{cor} \label{cor0101}
Let $R$ be a local ring.
If $R$ has a semidualizing complex that is neither dualizing nor free,
then the sequence  of Bass numbers 
$\{\mu^i_R(R)\}$ is bounded below by a linear polynomial in $i$ and hence
is not eventually constant.
\end{cor}

\begin{proof}
The assumption on $R$ yields a chain in $\s(\rhat)$
of the form $\cl{D'}\trin\cl \chat\trin\cl \rhat$, so the result follows from
Theorem~\ref{thm0101} using the equality
$\mu^i_R(R)=\mu^i_{\rhat}(\rhat)$.
\end{proof}

\section{Bounding Lengths of Chains of Semidualizing Complexes} \label{sec06}

In this section we use Lemma~\ref{l0501} to show how the Bass numbers of
$R$ in low degree can be used to bound the lengths of chains in $\s(R)$. 
The first two results 
contain Theorem~\ref{thm0002} from the introduction
and focus on the first two nonzero Bass numbers.
The results of this section are not  exhaustive.
Instead, they are meant to give a sampling  of applications of Lemma~\ref{l0501}.
For instance, the same technique can be used to give similar bounds
in terms of higher-degree Bass numbers.

\begin{thm} \label{prop0101}
Let $R$ be a local Cohen-Macaulay ring of depth $g$,
and let $h$ denote the number of prime factors of the integer
$\mu^g_R(R)$, counted with multiplicity. 
If $R$ has a chain of semidualizing modules of length $d$,
then $d\leq h\leq\mu^g_R(R)$.
In particular, if $\mu^g_R(R)$
is prime, then every semidualizing $R$-module
is either free or dualizing for $R$.
\end{thm}

\begin{proof}
By  Lemma~\ref{l0501}, the existence of a chain in $\s_0(R)$
of length $d$ yields a factorization
$I^R_R(t)=t^{g}P_1(t)\cdots P_{d}(t)$
where each $P_i(t)$ is a power series  with positive integer
coefficients 
and constant term $a_i\geq 2$. We then have
$$\mu^g_R(R)=a_1\cdots a_{d}$$
so the inequalities $d\leq h\leq\mu^g_R(R)$ 
follow from the basic properties of factorizations of integers.

Assume now that $\mu^g_R(R)$ is prime and let $C$ be a semidualizing 
$R$-module. 
The ring $\rhat$ has a dualizing module $D'$ by Fact~\ref{f0215'},
and Fact~\ref{f0216'} shows that there is a chain $\cl{D'}\tri\cl \chat\tri\cl \rhat$ in $\s(\rhat)$.
This chain must have length at most 1
since the Bass number $\mu^g_{\rhat}(\rhat)=\mu^g_R(R)$ is prime.
Hence, either $\chat\cong\rhat$ or $\chat\cong D'$.
From Fact~\ref{f0204'}, it follows that
the $R$-module $C$ is either free or dualizing for $R$.
\end{proof}

\begin{thm} \label{prop0103}
Let $R$ be a local ring of depth $g$.
If $R$ has a chain of semidualizing complexes of length $d$,
then $d\leq\mu^{g+1}_R(R)$. In particular, the set $\s(R)$ does not contain
arbitrarily long chains.
\end{thm}

\begin{proof}
Assume that $\s(R)$ contains a chain of length $d$. Lemma~\ref{l0501}
yields power series $P_1(t),\ldots,P_d(t)$ with positive integer
coefficients such that
\begin{equation} 
\label{prop0103d}
I^R_R(t)=t^{g}P_1(t)\cdots P_d(t).
\end{equation}
For each index $i$, write $P_i(t)=\sum_{j=0}^\infty a_{i,j}t^j$.
By calculating the degree $g+1$ coefficient in~\eqref{prop0103d}, we obtain the first equality
in the following sequence
\begin{equation*} 
\mu^{g+1}_R(R)
=\sum_{i=1}^{d}\frac{a_{1,0}\cdots a_{d,0}}{a_{i,0}}a_{i,1}\geq\sum_{i=1}^{d}a_{i,1}
\geq\sum_{i=1}^{d}1=d.
\end{equation*}
The inequalities are from the conditions
$a_{j,0},a_{i,1}\geq 1$.
\end{proof}

The next result gives an indication how other Bass numbers can also give
information about the chains in $\s(R)$.

\begin{prop} \label{prop0102}
Let $R$ be a local ring of depth $g$.
\begin{enumerate}[\quad\rm(a)]
\item \label{prop0102a}
If $\mu^i_R(R)\leq i-g$ for some index $i\geq g$, then every semidualizing $R$-complex
is either free or dualizing for $R$.
\item \label{prop0102b}
Assume that $R$ is Cohen-Macaulay and let $p$ be the smallest prime divisor of 
$\mu^g_R(R)$. 
If $\mu^i_R(R)< 2p+i-g-1$ for some index $i> g$, then every semidualizing $R$-module
is either free or dualizing for $R$.
\end{enumerate}
\end{prop}

\begin{proof}
We prove the contrapositive of each statement.
Assume that $R$ has a semidualizing complex that is neither free nor 
dualizing. The set $\s(\rhat)$ then has a chain 
$\cl D\trin\cl C\trin\cl R$, so Lemma~\ref{l0501}
yields power series $P_1(t),P_2(t)$ with positive integer
coefficients such that
$I^R_R(t)=t^{g}P_1(t) P_2(t)$.
Write $P_1(t)=\sum_{i=0}^\infty a_it^i$ and $P_2(t)=\sum_{i=0}^\infty b_it^i$.
It follows that, for each index $i\geq g$, we have
\begin{equation} \label{prop0102c}
\mu^i_R(R)
=\sum_{j=0}^{i-g}a_jb_{i-g-j}.
\end{equation}
\eqref{prop0102a}
Since each $a_j,b_j\geq 1$, the equation~\eqref{prop0102c} implies that
$$\mu^i_R(R)
=\sum_{j=0}^{i-g}a_jb_{i-g-j}
\geq\sum_{j=0}^{i-g}1
=i-g+1>i-g.
$$
\eqref{prop0102b}
Assume that $R$ is Cohen-Macaulay.
Lemma~\ref{l0501} implies that $a_0,b_0\geq p$. Assuming that $i>g$, 
equation~\eqref{prop0102c} reads
\begin{xxalignat}{3}
  &{\hphantom{\square}}& \mu^i_R(R)
&=\sum_{j=0}^{i-g}a_jb_{i-g-j}
\geq a_0+b_0+\sum_{j=1}^{i-g-1}1
\geq 2p+i-g-1. &&\qedhere
\end{xxalignat}
\end{proof}

The next example shows how Proposition~\ref{prop0102} applies to the
ring from~\ref{ex0101}.

\begin{ex} \label{ex0601}
Let $k$ be a field and set $R=k[\![X,Y]\!]/(X^2,XY)$. Then $R$ is a complete local ring
of dimension 1 and depth 0. 
From Example~\ref{ex0101}
we have $\mu^2_R(R)=2$, so Proposition~\ref{prop0102}
implies that $\s(R)=\{\cl R,\cl D\}$.
\end{ex}

We conclude this section with some questions that arise naturally from this work
and from the literature on Bass numbers, followed by some discussion.

\begin{question} \label{q}
Let $R$ be a local ring and  $C$  a non-free semidualizing $R$-complex.
\begin{enumerate}[\quad(a)]
\item \label{q1}
Must the sequence $\{\beta^R_i(C)\}$ eventually be strictly increasing?
\item \label{q1'}
Must the sequence $\{\beta^R_i(C)\}$ be nondecreasing?
\item \label{q2} 
Must the sequence $\{\beta_i^R(C)\}$ be unbounded?
\item \label{q3} 
Can the sequence $\{\beta_i^R(C)\}$ be bounded above by a polynomial in $i$?
\item \label{q4} 
Must the sequence $\{\beta_i^R(C)\}$ grow exponentially?
\item \label{q5} 
If $C$ is not dualizing for $R$, must the sequence $\{\mu^i_R(R)\}$ be strictly increasing?
\end{enumerate}
\end{question}

\begin{disc} \label{d2}
Question~\ref{q}\eqref{q1} relates to~\cite[Question 2]{christensen:iabbn}
where it is asked whether the Bass numbers of a non-Gorenstein local ring must
eventually be strictly increasing.
(Note that Example~\ref{ex0101} shows that they need not be always strictly increasing.)
If Question~\ref{q}\eqref{q1} is answered in the affirmative,
then so is~\cite[Question 2]{christensen:iabbn} since the Bass numbers of $R$
are given as the Betti numbers of the dualizing complex for $\rhat$.
Part~\eqref{q1'} is obviously similar to part~\eqref{q1},
and parts~\eqref{q2}--\eqref{q4} of Question~\ref{q} relate similarly to Question~\ref{q0001}.

Question~\ref{q}\eqref{q5} is a bit different. The idea here is that the existence of a semidualizing
$R$-complex that is not free and not dualizing provides a chain of length 2 in $\s(\rhat)$.
Hence, Lemma~\ref{l0501} gives a nontrivial factorization  $I^R_R(t)=t^gP_1(t)P_2(t)$
where each $P_i(t)=t^{m_i}P^R_{C^i}(t)$ for some non-free semidualizing 
$R$-complex $C^i$. If the coefficients of each $P_i(t)$ are strictly increasing,
then the coefficients of the product 
$I^R_R(t)=t^gP_1(t)P_2(t)$ are also strictly increasing. Note, however,
that the positivity of the coefficients of the $P_i(t)$ is not enough to ensure that
the coefficients of $I^R_R(t)$ are strictly increasing.
For instance, we have
$$(2+t+t^2+t^3+\cdots)(5+t+t^2+t^3+\cdots)
=10+7t+8t^2+9t^3+\cdots.$$
\end{disc}

\appendix

\section{Homological Algebra for Complexes} \label{sec01}

This appendix contains notation and useful facts about chain complexes for use
in Sections~\ref{sec02}--\ref{sec06}.
We do not attempt to explain every detail about
complexes that we use. For this, we recommend that the interested reader 
consult a text like~\cite{gelfand:moha} or~\cite{hartshorne:rad}.
Instead, we 
give heuristic explanations of the ideas coupled with explicit connections
to the corresponding notions for modules.
This way, the reader who is familiar with the homological
algebra of modules can get a feeling for the subject
and will possibly be motivated to investigate the subject more deeply.

\begin{defn} \label{d0111}
A \emph{chain complex of $R$-modules}, or \emph{$R$-complex} for short,
is a sequence of $R$-module homomorphisms
$$X=\cdots\xra{\partial^X_{i+1}}X_{i}\xra{\partial^X_{i}}X_{i-1}\xra{\partial^X_{i-1}}\cdots$$
such that $\partial^X_i\partial^X_{i+1}=0$ for each $i\in\bbz$.
The $i$th \emph{homology module} of an $R$-complex $X$ is 
the $R$-module
$\HH_i(X)=\ker(\partial^X_i)/\im(\partial^X_{i+1})$.
A \emph{morphism} of chain complexes $f\colon X\to Y$ is a sequence of $R$-module
homomorphisms $\{f_i\colon X_i\to Y_i\}_{i\in\bbz}$ 
making the following diagram commute
$$\xymatrix{
X\ar[d]_f
& \cdots\ar[r]^-{\partial^X_{i+1}}
&X_{i}\ar[r]^-{\partial^X_{i}}\ar[d]_{f_i}
&X_{i-1}\ar[r]^-{\partial^X_{i-1}}\ar[d]_{f_{i-1}}
&\cdots \\
Y
& \cdots\ar[r]^-{\partial^Y_{i+1}}
&Y_{i}\ar[r]^-{\partial^Y_{i}}
&Y_{i-1}\ar[r]^-{\partial^Y_{i-1}}
&\cdots 
}$$
that is, such that $\partial^Y_if_i=f_{i-1}\partial^X_i$ for all $i\in\bbz$.
A morphism $f\colon X\to Y$ induces an $R$-module homomorphism
$\HH_i(f)\colon\HH_i(X)\to\HH_i(Y)$ for each $i\in\bbz$.
The morphism $f$ is a \emph{quasiisomorphism} if the map $\HH_i(f)$
is an isomorphism for each $i\in\bbz$.
\end{defn}

\begin{notation} \label{n0101}
The category of $R$-complexes is denoted $\catc(R)$.
The category of $R$-modules is denoted $\catm(R)$.
Isomorphisms in  each of these categories are identified by the symbol $\cong$,
and quasiisomorphisms in $\catc(R)$
are identified by the symbol $\simeq$.

The derived category of $R$-complexes is denoted $\catd(R)$.
Morphisms in $\catd(R)$ are 
equivalence classes of diagrams of morphisms in $\catc(R)$.
Isomorphisms in $\catd(R)$ 
correspond to quasiisomorphisms in $\catc(R)$ and
are identified by the symbol $\simeq$.
\end{notation}

The connection between $\catd(R)$ and $\catm(R)$
comes from the following.

\begin{fact} \label{f0103}
Each $R$-module $M$ is naturally associated with an $R$-complex concentrated in
degree 0, namely the complex $0\to M\to 0$. 
We use the symbol $M$ to designate both the module and the associated
complex. With this notation we have
$$\HH_i(M)\cong\begin{cases}
M & \text{if $i=0$} \\ 0 & \text{if $i\neq 0$.}\end{cases}$$
This association gives rise to a full 
embedding of the module category $\catm(R)$ into the derived category $\catd(R)$.
In particular, for $R$-modules $M$ and $N$ we have
$M\cong N$ in $\catm(R)$ if and only if  $M\simeq N$ in $\catd(R)$.
\end{fact}

\begin{fact} \label{f0110}
Let $X$ and $Y$ be $R$-complexes. If $X\simeq Y$ in
$\catd(R)$, then we have $\HH_i(X)\cong\HH_i(Y)$ for all $i\in \bbz$.
The converse fails in general.
However, there is an
isomorphism $X\simeq \HH_0(X)$ in $\catd(R)$ if and only if $\HH_i(X)=0$ for all $i\neq 0$.
In particular, we have $X\simeq 0$ in $\catd(R)$ if and only if $\HH_i(X)=0$ for all $i\in\bbz$.
\end{fact}

The next invariants conveniently measure the homological position
of a complex.

\begin{defn} \label{d0108}
The \emph{supremum} and  \emph{infimum} of an $R$-complex $X$ are, respectively
$$\sup(X)=\sup\{i\in\bbz\mid\HH_i(X)\neq 0\}
\qquad\text{and}\qquad
\inf(X)=\inf\{i\in\bbz\mid\HH_i(X)\neq 0\}$$
with the conventions $\inf\emptyset=\infty$ and $\sup\emptyset=-\infty$.
\end{defn}

\begin{fact} \label{f0108}
Let $X$ be an $R$-complex.
If $X\not\simeq 0$,
then $-\infty\leq\inf (X)\leq\sup(X)\leq\infty$.
Also $\inf (X)=\infty$ if and only if $X\simeq 0$ if and only if $\sup(X)=-\infty$.
If $M\neq 0$ is an $R$-module, considered as an $R$-complex, then
$\inf(M)=0=\sup(M)$.
\end{fact}

The next construction allows us to ``shift'' a given $R$-complex,
which is useful, for instance, when we want the nonzero homology modules
in nonnegative degrees.

\begin{defn} \label{d0107}
Let $X$ be an $R$-complex.
For each integer $i$, the $i$th \emph{suspension}
or \emph{shift} of $X$ is the complex
$\shift^iX$ given by
$(\shift^iX)_j=X_{j-i}$
and $\partial^{\shift^iX}_j=(-1)^i\partial^X_{j-i}$.
\end{defn}

\begin{fact} \label{f0107}
If $X$ is an $R$-complex, then $\shift^iX$ is obtained by shifting $X$
to the left by $i$ degrees and multiplying the differential by $(-1)^i$. 
In particular, if $M$ is an $R$-module, then $\shift^iM$ is 
a complex that is concentrated
in degree $i$.
It is straightforward to show that
$\HH_j(\shift^iX)\cong\HH_{j-i}(X)$, and hence
$\inf(\shift^iX)=\inf(X)+i$
and
$\sup(\shift^iX)=\sup(X)+i$.
\end{fact}

For most of this investigation, we focus on $R$-complexes with only 
finitely many nonzero homology
modules, hence the next terminology. 

\begin{defn} \label{d0101}
An $R$-complex $X$ is \emph{bounded} if $X_i=0$ for $|i|\gg 0$.
It is \emph{homologically bounded below} if $\HH_i(X)=0$ for $i\ll 0$.
It is \emph{homologically bounded above} if $\HH_i(X)=0$ for $i\gg 0$.
It is \emph{homologically bounded} if $\HH_i(X)=0$ for $|i|\gg 0$.
It is \emph{homologically degreewise finite} if each homology module $\HH_i(X)$
is finitely generated.
It is \emph{homologically finite}
if the module $\HH(X)=\oplus_{i\in\bbz}\HH_i(X)$
is finitely generated.
\end{defn}

The next fact summarizes elementary translations of these definitions.

\begin{fact} \label{f0001}
An $R$-complex $X$ is homologically bounded below
if $\inf(X)>-\infty$.
It is homologically bounded above 
if $\sup(X)<\infty$.
Hence, it is homologically bounded
if $\inf(X)>-\infty$ and $\sup(X)<\infty$, that is, if it is homologically bounded both 
above and below.
The complex $X$ is homologically finite
if it
is homologically both degreewise finite and bounded.

Each of the properties defined in~\ref{d0101} is invariant under shift.
For instance, an $R$-complex $X$ is homologically finite if and only if
some (equivalently, every) shift $\shift^iX$ is homologically finite; see Fact~\ref{f0107}.
\end{fact}

For modules, many of these notions are trivial:

\begin{fact} \label{f0105}
An $R$-module $M$ is always homologically bounded as an $R$-complex.
It is homologically  finite as an $R$-complex
if and only if it is finitely generated. 
\end{fact}

As with modules, there are various useful types of resolutions of $R$-complexes.

\begin{defn} \label{d0103}
Let $X$ be an $R$-complex.
An \emph{injective resolution}\footnote{Note that our injective resolutions
are bounded above by definition. There are notions of 
injective (and projective) resolutions for unbounded
complexes, but we do not need them here. The interested reader should 
consult~\cite{avramov:hdouc} for information on these more general constructions.} 
of  $X$ is an
$R$-complex $J$ such that $X\simeq J$ in $\catd(R)$,
each $J_i$ is injective, and $J_i=0$ for $i\gg 0$.
The complex $X$ has \emph{finite injective dimension}
if it has an injective resolution $J$ such that $J_i=0$ for $i\ll 0$.
More specifically, the injective dimension of  $X$ is
$$\id_R(X)=\inf\{\sup\{i\in\bbz\mid J_{-i}\neq 0\}\mid
\text{$J$ is an injective resolution of $X$}\}.$$
Dually, a \emph{free resolution} of  $X$ is an
$R$-complex $F$ such that $F\simeq X$ in $\catd(R)$,
each $F_i$ is free, and $F_i=0$ for $i\ll 0$.
The complex $X$ has \emph{finite projective dimension}\footnote{Since
the ring $R$ is local, every projective $R$-module is free. For this reason,
we focus on free resolutions instead of projective ones. On the other hand, tradition dictates
that the corresponding homological dimension is the ``projective dimension''
instead of the possibly confusing (though, potentially liberating) ``free dimension''.}
if it has a free resolution $F$ such that $F_i=0$ for $i\gg 0$.
More specifically, the projective dimension of  $X$ is
$$\pd_R(X)=\inf\{\sup\{i\in\bbz\mid F_{i}\neq 0\}\mid
\text{$F$ is a free resolution of $X$}\}.$$
A free resolution $F$ of $X$ is \emph{minimal}\footnote{There is
also a notion of minimal injective resolutions of complexes, but 
it is slightly more complicated, and we do not need it here.}
if for each index $i$, the module $F_i$ is finitely generated
and $\im(\partial^F_i)\subseteq\m F_{i-1}$.
\end{defn}

For modules, the notions from~\ref{d0103} are the familiar ones.

\begin{fact} \label{f0109}
Let $M$ be an $R$-module. An injective resolution of $M$ as an $R$-module,
in the traditional sense of an exact sequence of the form
$$0\to M\to J_0\xra{\partial_0} J_{-1}\xra{\partial_{-1}}\cdots$$
where each $J_i$ is injective, gives rise to an injective resolution of 
$M$ as an $R$-complex:
$$0\to J_0\xra{\partial_0} J_{-1}\xra{\partial_{-1}}\cdots.$$
Conversely, every injective resolution of $M$ as an $R$-complex 
gives rise to an injective resolution of $M$ as an $R$-module,
though one has to work a little harder. 
Accordingly, the injective dimension of $M$ as an $R$-module
equals the injective dimension of $M$ as an $R$-complex.
Similar comments apply 
to free resolutions and projective dimension.
\end{fact}

The next fact summarizes basic properties about existence of these resolutions.

\begin{fact} \label{f0101}
Let $X$ be an $R$-complex.
Then  $X$ has a free resolution if and only if it is homologically bounded below;
when these conditions are met, 
it has a free resolution $F$ such that $F_i=0$ for all $i<\inf(X)$;
see~\cite[(2.11.3.4)]{avramov:dgha}
or~\cite[(6.6.i)]{felix:rht} or~\cite[(2.6.P)]{foxby:hacr}. 
Dually, 
the complex $X$ has an injective resolution if and only if it is homologically bounded above;
when these conditions are satisfied, 
it has an injective resolution $J$ such that $J_i=0$ for all $i>\sup(X)$.
If $X$ is homologically both degreewise finite
and bounded below, then it has a minimal free resolution $F$,
and one has $F_i=0$ for all $i<\inf(X)$;
see~\cite[Prop.\ 2]{apassov:cdgm}
or~\cite[(2.12.5.2.1)]{avramov:dgha}. 
\end{fact}

These invariants interact with the shift operator as one might expect:

\begin{fact} \label{f0101'}
It is straightforward to show that,
if $X$ is an $R$-complex and $i$ is an integer, then
$\id_R(\shift^iX)=\id_R(X)-i$ and
$\pd_R(\shift^iX)=\pd_R(X)+i$.
\end{fact}

The next constructions extend Hom and tensor product to the category $\catc(R)$.

\begin{defn} \label{d0105}
Let $X$ and $Y$ be $R$-complexes. The 
\emph{tensor product complex} $X\otimes_RY$
and
\emph{homomorphism complex}
$\Hom_R(X,Y)$ 
are defined by the formulas
\begin{align*}
(X\otimes_RY)_i&=\coprod_{j\in\bbz}X_j\otimes_RY_{i-j} \\
\partial^{X\otimes_RY}_i(\{x_j\otimes y_{i-j}\}_{j\in\bbz})
&=\{\partial^X_j(x_j)\otimes y_{i-j}+(-1)^{j-1}x_{j-1}\otimes\partial^Y_{i-j+1}(y_{i-j+1})\}\\
\Hom_R(X,Y)_i&=\prod_{j\in\bbz}\hom_R(X_j,Y_{j+i}) \\
\partial^{\hom_R(X,Y)}_i(\{\phi_j\}_{j\in\bbz})
&=\{\partial^Y_{j+i}\phi_j-(-1)^i\phi_{j-1}\partial^X_{j}\}.
\end{align*}
\end{defn}

When one of the complexes in this definition is a module, 
the resulting complexes have the form one should expect:

\begin{fact} \label{f0104}
Let $X$ be an $R$-complex and $M$ an $R$-module. 
The complexes 
$X\otimes_RM$, $M\otimes_RX$ and $\hom_R(M,X)$
are exactly the complexes you would expect, namely
\begin{align*}
X\otimes_R M=
&\cdots\xra{\partial^X_{i+1}\otimes M}X_{i}\otimes M
\xra{\partial^X_{i}\otimes M}X_{i-1}\otimes M
\xra{\partial^X_{i-1}\otimes M}\cdots \\
M\otimes_R X=
&\cdots\xra{M\otimes\partial^X_{i+1}}M\otimes X_{i}
\xra{M\otimes\partial^X_{i}}M\otimes X_{i-1}
\xra{M\otimes\partial^X_{i-1}}\cdots \\
\hom_R(M,X) =\\
&\hspace{-1.5cm}\cdots\xra{\hom(M,\partial^X_{i+1})} \hom (M,X_{i})
\xra{\hom(M,\partial^X_{i})}  \hom (M,X_{i-1})\xra{\hom(M,\partial^X_{i-1})} 
\cdots.
\intertext{On the other hand, the complex $\hom (X,M)$ 
has the form you would expect, but the differentials differ by a sign: }
\hom_R (X,M)=\\
&\hspace{-1.5cm}\cdots\xra{(-1)^i\hom (\partial^X_{i},M)} \hom (X_{i},M)
\xra{(-1)^{i+1}\hom (\partial^X_{i+1},M)}  \hom (X_{i-1},M)\cdots.
\end{align*}
Note that this sign difference does not change the homology since it 
changes neither the kernels nor the images of the respective maps.
\end{fact}

Here are some standard isomorphisms we shall need.

\begin{fact} \label{f3112}
Let $X$, $Y$ and $Z$ be  $R$-complexes.
The following natural isomorphisms are straightforward to verify,
using the counterparts for modules in the first five, and using the definition in the last:
\begin{align}
\hom_R(R,X)
&\cong X \label{f3112b} \tag{cancellation}\\
X\otimes_RY
&\cong Y\otimes_RX  \label{f3112c} \tag{commutativity}\\
\hom_R(X\oplus Y,Z)
&\cong\hom_R(X,Z)\oplus\hom_R(Y,Z) \tag{additivity} \\
\hom_R(X,Y\oplus Z)
&\cong\hom_R(X,Z)\oplus\hom_R(X,Y) \tag{additivity} \\
\hom_R(X\otimes_RY,Z)
&\cong \hom_R(X,\hom_R(Y,Z))  \label{f3112e} \tag{adjointness} \\
\hom_R(\shift^iX,\shift^jY)
&\cong\shift^{j-i}\hom_R(X,Y). \tag{shift}
\intertext{Let $S$ be a flat $R$-algebra. If each $R$-module $X_i$ is
finitely generated and $X_i=0$ for $i\ll 0$, then}
\hom_S(S\otimes_RX,S\otimes_RY)
&\cong S\otimes_R\hom_R(X,Y). \tag{base-change}
\end{align}
\end{fact}

Bounded complexes yield bounded
homomorphism and tensor product complexes. More specifically, the next
fact follows straight from the definitions.

\begin{fact} \label{f0112}
Let $X$ and $Y$ be  $R$-complexes.
If $X_i=0=Y_j$ for all $i< m$ and all $j< n$, then
$(X\otimes_RY)_i=0$ for all $i< m+n$.
If $X_i=0=Y_j$ for all $i< m$ and all $j> n$, then
$\hom_R(X,Y)_i=0$ for all $i> n-m$.
\end{fact}

Here is the notation for derived functors in 
the derived category
$\catd(R)$.

\begin{notation} \label{d0106}
Let $X$ and $Y$ be $R$-complexes. The 
left-derived tensor product
and right-derived homomorphism complexes in $\catd(R)$ are denoted
$X\lotimes_RY$ and $\rhom_R(X,Y)$.
\end{notation}

The complexes $X\lotimes_RY$ and $\rhom_R(X,Y)$ are computed
using the same rules as for computing Tor and Ext of modules:

\begin{fact} \label{f0106}
Let $X$ and $Y$ be $R$-complexes. 
If $F$ is a free resolution of $X$
and $G$ is a free  resolution of $Y$,
then
$$X\lotimes_RY\simeq F\otimes_RY\simeq F\otimes_RG\simeq X\otimes_RG.$$
If $F$ is a free resolution of $X$
and $I$ is an injective resolution of $Y$,
then
$$\rhom_R(X,Y)\simeq \hom_R(F,Y)\simeq\hom_R(F,I)\simeq\hom_R(X,I).$$
It follows that, if $M$ and $N$ are $R$-modules, then
$\tor_i^R(M,N)\cong\HH_{i}(M\lotimes_RN)$
and $\ext^i_R(M,N)\cong\HH_{-i}(\rhom_R(M,N))$
for every integer $i$.
\end{fact}

The next isomorphisms follow from Fact~\ref{f3112} using appropriate resolutions.

\begin{fact} \label{f4112}
If $X$, $Y$ and $Z$ are  $R$-complexes, then there are isomorphisms in $\catd(R)$
\begin{align}
\rhom_R(R,X)
&\simeq X \label{f4112b} \tag{cancellation}\\
X\lotimes_RY
&\simeq Y\lotimes_RX  \label{f4112c} \tag{commutativity}\\
\rhom_R(X\oplus Y,Z)
&\simeq\rhom_R(X,Z)\oplus\rhom_R(Y,Z) \tag{additivity} \\
\rhom_R(X,Y\oplus Z)
&\simeq\rhom_R(X,Z)\oplus\rhom_R(X,Y) \tag{additivity} \\
\rhom_R(X\lotimes_RY,Z)
&\simeq \rhom_R(X,\rhom_R(Y,Z))  \label{f4112e} \tag{adjointness}\\
\rhom_R(\shift^iX,\shift^jY)
&\simeq\shift^{j-i}\rhom_R(X,Y). \tag{shift}
\intertext{Let $S$ be a flat $R$-algebra. If 
$X$ is homologically both degreewise finite and bounded below, then}
\rhom_S(S\lotimes_RX,S\lotimes_RY)
&\simeq S\lotimes_R\rhom_R(X,Y). \tag{base-change}
\end{align}
\end{fact}

The following homological bounds are consequences of Fact~\ref{f0112}.

\begin{fact} \label{f0111}
Let $X$ and $Y$ be homologically bounded below $R$-complexes. 
Let $F$ and $G$ be free resolutions $X$ and $Y$, respectively,
such that $F_i=0$ for $i<\inf (X)$ and $G_i=0$ for $i<\inf(Y)$.
It follows that, for $i<\inf(X)+\inf(Y)$, we have
$$\HH_i(X\lotimes_RY)\cong\HH_i(F\otimes_RG)=0$$
and hence $\inf(X\lotimes_RY)\geq\inf(X)+\inf(Y)$.
Furthermore, the right exactness of tensor product yields the second isomorphism
in the next sequence
$$\HH_{\inf(X)+\inf(Y)}(X\lotimes_RY)
\cong \HH_{\inf(X)+\inf(Y)}(F\otimes_RG)
\cong \HH_{\inf(X)}(X)\otimes_R \HH_{\inf(Y)}(Y).$$
This corresponds to the well-known formula $\tor_0^R(M,N)\cong M\otimes_RN$
for modules $M$ and $N$.
If $\HH_{\inf(X)}(X)$ and $\HH_{\inf(Y)}(Y)$ are both finitely generated,
e.g., if $X$ and $Y$ are both homologically degreewise finite,
then Nakayama's Lemma implies that
$$\HH_{\inf(X)+\inf(Y)}(X\lotimes_RY)
\cong \HH_{\inf(X)}(X)\otimes_R \HH_{\inf(Y)}(Y)\neq 0$$
and thus $\inf(X\lotimes_RY)=\inf(X)+\inf(Y)$.
Note that this explicitly uses the assumption that $R$ is local.

A similar argument shows that, when $Z$ is homologically bounded above,
then the complex $\rhom_R(X,Z)$ is homologically bounded above:
there is an inequality
$\sup(\rhom_R(X,Z))\leq\sup (Z)-\inf (X)$ and an isomorphism
$$\HH_{\sup (Z)-\inf (X)}(\rhom_R(X,Z))\cong\hom_R(\HH_{\inf(X)}(X),\HH_{\sup(Z)}(Z)).$$
\end{fact}

The next fact is a derived category
version of  the finite generation of Ext and Tor of finitely generated modules.
It essentially follows from~\ref{f0106}.

\begin{fact} \label{f0114}
Let $X$ and $Y$ be  $R$-complexes that are homologically both
degreewise finite and bounded below.
Let $F$ and $G$ be free resolutions of $X$ and $Y$, respectively,
such that each $F_i$ and $G_i$ is  finitely generated.
Then $F\otimes_RG$ is a free resolution of $X\lotimes_RY$, and
each $R$-module
$(F\otimes_RG)_i$ is  finitely  generated.
In particular, the complex $X\lotimes_RY$ is homologically both
degreewise finite and bounded below.
If $F$ and $G$ are minimal,
then
$F\otimes_RG$ is a minimal free resolution of $X\lotimes_R Y$.

It takes a little more work to show that, if $Z$ is homologically both
degreewise finite and bounded above, then 
the $R$-complex
$\rhom_R(X,Z)$ is homologically both degreewise finite and bounded above.
\end{fact}

Here are some homological invariants that are familiar for modules.

\begin{defn} \label{d0104}
Let $X$ be a homologically finite $R$-complex.
The $i$th \emph{Bass number} of  $X$ is 
the integer $\mu^i_R(X)=\rank_k(\HH_{-i}(\rhom_R(k,X)))$, and
the \emph{Bass series} of $X$ is the formal Laurent series
$I^X_R(t)=\sum_{i\in\bbz}\mu^i_R(X)t^i$.
The $i$th \emph{Betti number} of  $X$ is 
the integer $\beta_i^R(X)=\rank_k(\HH_{i}(k\lotimes_RX))$, and
the \emph{Poincar\'{e} series} of $X$ is the formal Laurent series
$P_X^R(t)=\sum_{i\in\bbz} \beta_i^R(X)t^i$.
\end{defn}

\begin{fact} \label{f0113}
If $M$ is an $R$-module, then we have
$\mu^i_R(M)=\rank_k(\ext^i_R(k,M))$
and
$\beta^R_i(M)=\rank_k(\tor^R_i(k,M))=\rank_k(\ext^i_R(M,k))$.
\end{fact}

We conclude with useful formulas for the Poincar\'e and Bass series of, respectively,
derived tensor products and derived homomorphism complexes.

\begin{fact} \label{f0102}
Let $X$ and $Y$ be  $R$-complexes that are homologically both
degreewise finite and bounded below.
If $F$ is a minimal free resolution of $X$,
then $\beta_i^R(X)=\rank_R(F_i)$ for all $i\in\bbz$.
(Indeed the complex $k\otimes_RF$
has zero differential, and hence
$$\HH_i(k\lotimes_RX)
\cong \HH_i(k\otimes_RF)
\cong (k\otimes_RF)_i
\cong k\otimes_RF_i.
$$
The $k$-vector space rank of this module is precisely $\rank_R(F_i)$.)
Combining this with
Fact~\ref{f0114}, we conclude that
$$P_{X\lotimes_RY}^R(t)=P_X^R(t)P_Y^R(t).$$
Furthermore, the equality $\beta_i^R(X)=\rank_R(F_i)$ for all $i\in\bbz$
implies that $P^R_X(t)=0$ if and only if $F=0$, that is, if and only if $X\simeq 0$.
See also Fact~\ref{f0111}.

Given an $R$-complex $Z$ that is homologically both
degreewise finite and bounded above, a 
different argument
yields the next formula
$$I^{\rhom_R(X,Z)}_R(t)=P^R_X(t)I^Z_R(t).$$
Furthermore, we have $I^Z_R(t)=0$ if and only if $Z\simeq 0$.
See~\cite[(1.5.3)]{avramov:rhafgd}.
\end{fact}

\section*{Acknowledgments}
I am grateful to Lars W.\ Christensen, Jim Coykendall,
David Jorgensen and Graham Leuschke for stimulating conversations
about this research, 
to Youngsu Kim 
and the referee for helpful comments,
and to Fatemeh  Mohammadi for noticing errors in an earlier version.

%\bibliography{../new}
\providecommand{\bysame}{\leavevmode\hbox to3em{\hrulefill}\thinspace}
\providecommand{\MR}{\relax\ifhmode\unskip\space\fi MR }
\providecommand{\MRhref}[2]{%
  \href{http://www.ams.org/mathscinet-getitem?mr=#1}{#2}
}
\providecommand{\href}[2]{#2}

\end{document}